\documentclass[12pt]{amsart}

\usepackage[foot]{amsaddr}
\usepackage[english]{babel}
\usepackage[
pdftex,hyperfootnotes]{hyperref}
\usepackage[usenames,dvipsnames,svgnames,table,x11names]{xcolor}
\usepackage{pgfplots}
\usepackage{tikz}
\usepackage{tikz-3dplot}
\usetikzlibrary{calc,shadows,shapes.callouts,shapes.geometric,shapes.misc,positioning,patterns,decorations.pathmorphing,decorations.markings,decorations.fractals,decorations.pathreplacing,shadings,fadings}

\usepackage[utf8]{inputenc}

\usepackage[total={16cm,20cm}, top=4.5cm, left=2.5cm]{geometry}
\usepackage{rotating,multirow,pdflscape}
\usepackage{amssymb,latexsym,amsmath,amsthm}
\usepackage{mathrsfs}
\usepackage{mathtools,arydshln,mathdots}
\usepackage{bigints}
\usepackage{comment}

\makeatletter
\renewcommand*\env@matrix[1][*\c@MaxMatrixCols c]{%
 \hskip -\arraycolsep
 \let\@ifnextchar\new@ifnextchar
 \array{#1}}
\makeatother

\mathtoolsset{centercolon}
\usepackage[table]{xcolor}
\usepackage[toc,page]{appendix}

\usepackage{tocvsec2}

\usepackage{Baskervaldx}
\usepackage{newtxmath }
\newtheorem{coro}{Corollary}

\newtheorem{defi}{Definition}
\newtheorem{teo}{Theorem}
\newtheorem{pro}{Proposition}

\newcommand{\ii}{\operatorname{i}}

\renewcommand{\d}{\operatorname{d}}
\newcommand{\Exp}[1]{\operatorname{e}^{#1}}

\newcommand{\C}{\mathbb{C}}

\newcommand{\N}{\mathbb{N}}
\newcommand{\R}{\mathbb{R}}

\makeatletter
\def\@settitle{\begin{center}%
 \baselineskip14\p@\relax
 \bfseries
 \uppercasenonmath\@title
 \@title
 \ifx\@subtitle\@empty\else
 \\[1ex]\uppercasenonmath\@subtitle
 \footnotesize\mdseries\@subtitle
 \fi
 \end{center}%
}
\def\subtitle#1{\gdef\@subtitle{#1}}
\def\@subtitle{}
\makeatother

\title[Matrix Biorthogonality]{Riemann--Hilbert Problem for the \\
Matrix Laguerre Biorthogonal Polynomials: Eigenvalue Problems and the Matrix Discrete Painlevé IV}



\author[A Branquinho]{Amílcar Branquinho$^\dag$}
\address{$^\dag$Departamento de Matem\'atica,
Universidade de Coimbra, 3001-454 Coimbra, Portugal}
\email{ajplb@mat.uc.pt}

\author[A Foulqué]{Ana Foulqui\'e Moreno$^\maltese$}
\address{$^\maltese$Departamento de Matemática, Universidade de Aveiro, 3810-193 Aveiro, Portugal}
\email{foulquie@ua.pt}

\author[M Mañas]{Manuel Mañas$^\ddag$}
\address{$^\ddag$Departamento de Física Teórica, Universidad Complutense de Madrid, 28040-Madrid, Spain \&
Instituto de Ciencias Matematicas (ICMAT), Campus de Cantoblanco UAM, 28049-Madrid, Spain}
\email{manuel.manas@ucm.es}

\thanks{$^\dag$Acknowledges Centro de Matemática da Universidade de Coimbra (CMUC) -- UID/MAT/00324/2019, funded by the Portuguese Government through FCT/MEC and co-funded by the European Regional Development Fund through the Partnership Agreement PT2020}

\thanks{$^\maltese$Acknowledges CIDMA Center for Research and Development in Mathematics and Applications (University of Aveiro) and the 
Portuguese Foundation for Science and Technology (FCT)
within project UID/MAT/04106/2019}

\thanks{$^\ddag$Thanks financial support from the Spanish ``Agencia Estatal de Investigación'' research project [PGC2018-096504-B-C33], \emph{Ortogonalidad y Aproximación: Teoría y Aplicaciones en Física Matemática}.}

\begin{document}

\maketitle

\begin{abstract}
In this paper the Riemann--Hilbert problem, with jump supported on a appropriate curve on the complex plane with a finite endpoint at the origin, is used for the study of corresponding matrix biorthogonal polynomials associated with Laguerre type matrices of weights ---which are constructed in terms of a given matrix Pearson equation. First and second order differential systems for the fundamental matrix, solution of the mentioned Riemann--Hilbert problem are derived.
An explicit and general example is presented to illustrate the theoretical results of the work.
Related matrix eigenvalue problems for second order matrix differential operators and non-Abelian extensions of a family of discrete Painlev\'e~IV equations are discussed.
\end{abstract}

\section{Introduction}

Krein~\cite{Krein1,Krein2} was the first to discuss matrix extensions of real orthogonal polynomials, some relevant papers that appear afterwards on this subject are~\cite{bere},~\cite{geronimo} and more recently~\cite{nikishin}. 
The Russian mathematicians Aptekarev and Nikishin~\cite{nikishin} made a remarkable finding: 
for a kind of discrete Sturm--Liouville operators they solved the scattering problem and proved that the
polynomials that satisfy 
\begin{align*}
xP_{k}(x)&=A_{k}P_{k+1}(x)+B_{k}P_{k}(x)+A_{k-1}^{*}P_{k-1}(x),& k&=0,1,\ldots,
\end{align*}
are orthogonal with respect to a positive definite matrix measure,~i.e. they derived a matrix Favard theorem.
Later, it was found that matrix orthogonal polynomials (MOP) sometimes satisfy properties as do the classical orthogonal polynomials.
For example, for matrix versions of Laguerre, Hermite and Jacobi polynomials, the scalar-type Rodrigues' formula
\cite{duran_1,duran20052} and a second order differential equation~\cite{borrego,duran_3,duran2004} has been discussed. It also has been proven in~\cite{duran_5} that operators of the form
$D={\partial}^{2}F_{2}(t)+{\partial}^{1}F_{1}(t)+{\partial}^{0}F_{0}$ have as eigenfunctions different infinite families of~MOP's. 
In ~\cite{cum,carlos2} matrix extensions of the generalized polynomials considered in~\cite{adler-van-moerbeke,adler-vanmoerbeke-2} were studied.
Recently, in~\cite{nuevo}, the Christoffel transformation to matrix orthogonal polynomials in the real line (MOPRL)
have been extended and a new matrix Christoffel formula was obtained. 
Finally, in~\cite{nuevo2,nuevo3} more general transformations ---of Geronimus and Uvarov type--- where also considered.

Fokas, Its and Kitaev~\cite{FIK} found, in the context of 2D quantum gravity, that certain Riemann--Hilbert problem was solved in terms of orthogonal polynomials in the real line (OPRL). They found that the solution of a $2\times 2$
Riemann--Hilbert problem can be expressed in terms of orthogonal polynomials in the real line and its Cauchy transforms. Later, Deift and Zhou combined these ideas with a non-linear steepest descent analysis in a series of papers~\cite{deift1,deift2,deift3,deift4} which was the seed for a large activity in the field. To mention just a few relevant results let us cite the study of strong asymptotic with applications in random matrix theory~\cite{deift1,deift5}, the analysis of determinantal point processes~\cite{daems2,daems3,kuijlaars2,kuijlaars3}, orthogonal Laurent polynomials~\cite{McLaughlin1, McLaughlin2} and Painlev\'e equations~\cite{dai,kuijlaars4}.

 Recursion coefficients for orthogonal polynomials and its properties is a subject of current interest. See~\cite{VAssche,VAssche2} for a review on how the form of the weight and its properties translates to the recursion coefficients. Freud~\cite{freud0} has studied weights in $\R$ of exponential variation $w(x)=|x|^\rho\exp(-|x|^m)$, $\rho>-1$ and $m>0$. When $m=2,4,6$ he constructed relations among them as well as determined its asymptotic behavior. The role of the discrete Painlev\'e~I in this context was discovered later by Magnus~\cite{magnus}. For a weight of the form $w(\theta)=\exp(k\cos\theta)$, $k\in\mathbb R$, on the unit circle it was 
 found~\cite{Periwal0,Periwal} the discrete Painlev\'e~II equation for the recursion relations of the corresponding orthogonal polynomials, see also~\cite{hisakado} for a connection with the Painlev\'e~III equation.
 The discrete Painlev\'e~II was found in~\cite{baik0} using the Riemann--Hilbert problem given in~\cite{baik}, see also~\cite{tracy}. For a nice account of the relation of these discrete Painlev\'e equations and integrable systems see~\cite{cresswell}, and for a survey on the subject of differential and discrete Painlev\'e equations (cf.~\cite{Clarkson_1}). We also mention the recent paper~\cite{clarkson} where
a discussion on the relationship between the recurrence coefficients of orthogonal polynomials with respect to a semiclassical Laguerre weight and classical solutions of the fourth Painlev\'e equation can be found. Also, in~\cite{Clarkson_2} the solution of the discrete alternate Painlev\'e equations is presented in terms of the Airy~function.

In~\cite{CM} the Riemann--Hilbert problem
for this matrix situation and the appearance of non-Abelian discrete versions of Painlev\'e~I were explored, showing singularity confinement~\cite{CMT}, see also~\cite{GIM}. The
singularity analysis for a matrix discrete version of the Painlev\'e~I
equation was performed. It was found that the singularity 
confinement holds 
generically, i.e. in the whole space of parameters except possibly for algebraic subvarieties. 
The situation was considered in~\cite{Cassatella_3} for the matrix extension of the Szeg\H{o} polynomials in the unit circle and corresponding non-Abelian versions discrete Painlev\'e~II equations.


In~\cite{BFM} we have discussed matrix biorthogonal polynomials with matrix of weights $W(z)$ such that
\begin{itemize}
 \item The support of $W(z)$ is a non-intersecting smooth curve on the complex plane with no finite end points,~i.e. its end points occur at $\infty$.
 \item Weight matrix entries were, in principle, H\"older
continuous, and eventually requested to have holomorphic extensions to the complex plane.
 \item The matrix of weights $ W(z)$ is regular,~i.e., 
 $\det \big[ W_{j+k} \big]_{j,k=0, \ldots n}\not = 0 $, $n \in \N:=\{0,1,\ldots\} $, where the \emph{moment of order $n $}, $W_n$, associated with $ W$ is, for each $n\in\N$, given~by,
$\displaystyle W_n := \frac{1} {2\pi\ii} \int_\gamma z^n W (z) \d z$.
\end{itemize}
We obtained Sylvester systems of differential equations for the orthogonal polynomials and its second kind functions, directly from a Riemann--Hilbert problem, with jumps supported on appropriate curves on the complex plane. We considered a Sylvester type differential Pearson equation for the matrix of weights.
We also studied whenever the orthogonal polynomials and its second kind functions are solutions of a second order linear differential operators with matrix eigenvalues. This was done by stating an appropriate boundary value problem for the matrix of weights.
In particular, special attention was paid to non-Abelian Hermite biorthogonal polynomials in the real line, understood as those whose matrix of weights is a solution of a Sylvester type Pearson equation with given matrices of degree one polynomials coefficients. We also found nonlinear equations for the matrix coefficients of the corresponding three term relations, which extend to the non-commutative case the discrete Painlevé I and the alternate discrete Painlevé I equations.

In this paper we do a similar study but with more relaxed conditions, namely of Laguerre~type.
\begin{defi}[Laguerre type Matrix of weights] \label{def:laguerre weights}
We say that a regular matrix of weights $W= 
\begin{bsmallmatrix}
W^{(1,1)} & \cdots & W^{(1,N)} \\
\vdots & \ddots & \vdots \\
W^{(N,1)} & \cdots & W^{(N,N)} 
\end{bsmallmatrix}
\in \C^{ N\times N}$ is of Laguerre type if
\begin{itemize}
\item The support of 
 $W(z)$ 
is a non self-intersecting smooth curve on the complex plane with an end point at $0$ and the other end point at $\infty$, and such that it intersects the circles $|z|=R$, $R \in \R^+$, once and only once ({i.e.}, it can be taken as a determination curve for $\arg z$).
\item The entries $W^{(j,k)}$ of the matrix measure $W$ can be written as
\begin{align}\label{eq:the_weights}
W^{(j,k)}(z)& = \sum_{m \in I_{j,k}} h_m(z) z^{\alpha_{m}} \log^{p_{m}} z,& z&\in\gamma,
\end{align} 
where $I_{j,k}$ denotes a finite set of indexes, $ \alpha_{m} > -1$, $p_{m} \in \mathbb{N} \cup \{ 0 \} $ and $h_m(z)$ is H\"older continuous, bounded and non-vanishing on $\gamma$. 
Here the determination of logarithm and the powers are taken along $\gamma$. We will request, in the development of the theory, that the functions $h_m$ have a holomorphic extension to the whole complex plane.
\end{itemize}
\end{defi}

In this work, for the sake of simplicity, the finite end point of the curve $\gamma$ is taken at the origin, $c=0$, with no loss of generality, as a similar arguments apply for $c\neq 0$. In~\cite{duran2004} different examples of Laguerre matrix weights for the matrix orthogonal polynomials on the real line are studied.

\subsection{Matrix biorthogonal polynomials }

Given a Laguerre type matrix of weights as in Definition~\ref{def:laguerre weights} we introduce corresponding \emph{sequences of matrix monic polynomials}, the \emph{sequence of left matrix orthogonal polynomials}
$\big\{ {P}_n^{\mathsf L} (z)\big\}_{n\in\N} $ and the \emph{sequence of right matrix orthogonal polynomials} $\big\{ P_n^{\mathsf R}(z)\big\}_{n\in\N} $ characterized by the conditions,
\begin{align} \label{eq:ortogonalidadL}
\frac 1 {2\pi\ii} \int_\gamma {P}_n^{\mathsf L} (z) W (z) z^k \d z
= \delta_{n,k} C_n^{-1} , \\
\label{eq:ortogonalidadR}
\frac 1 {2\pi\ii} \int_\gamma z^k W (z) {P}_n^{\mathsf R} (z) \d z
= \delta_{n,k} C_n^{-1} , 
\end{align}
for $k = 0, 1, \ldots , n $ and $n \in \mathbb N$,
where $C_n $ is an nonsingular matrix.
The matrix of weights $W(z)$ induces a sesquilinear form in the set of matrix 
polynomials $\C^{N\times N}[z]$ given by
\begin{align}\label{eq:sesquilinear}
\langle P,Q \rangle_W :=\frac 1 {2\pi\ii} \int_\gamma {P} (z) W (z) Q(z) \d z,
\end{align}
for which $\big\{ P_n^{\mathsf L}(z)\big\}_{n\in\N}$ and $\big\{ P_n^{\mathsf R}(z)\big\}_{n\in\N}$ are biorthogonal 
\begin{align*} 
\big\langle P_n^{\mathsf L}, {P}_m^{\mathsf R} \big\rangle_ W
& = \delta_{n,m} C_n^{-1}, & n , m & \in \mathbb N .
\end{align*}
As the polynomials are chosen to be monic, we can write
\begin{align*}
{P}_n^{\mathsf L} (z) &= I_N z^n + p_{\mathsf L,n}^1 z^{n-1} + p_{\mathsf L,n}^2 z^{n-2} + \cdots + p_{\mathsf L,n}^n , \\
{P}_n^{\mathsf R} (z) &= I_N z^n + p_{\mathsf R,n}^1 z^{n-1} + p_{\mathsf R,n}^2 z^{n-2} + \cdots + p_{\mathsf R,n}^n , 
\end{align*}
with matrix coefficients $p_{\mathsf L, n}^k , p_{\mathsf R, n}^k \in 
\C^{N\times N} $, $k = 0, \ldots, n $ and $n \in \mathbb N $ (imposing that $p_{\mathsf L,n}^0 = p_{\mathsf R,n}^0=I $, $n \in \mathbb N $). Here $I_N\in\C^{N\times N}$ denotes the identity matrix.

We define the \emph{sequence of second kind matrix functions} by
\begin{align*} 
Q^{\mathsf L}_n (z) &
:
= \frac1{ 2 \pi \ii} \int_\gamma \frac{P^{\mathsf L}_n (z')}{z'-z} { W (z')} \d z^\prime , 
\\ 
{Q}_n^{\mathsf R} (z) &
:
= \frac{1}{2\pi \ii} \int_\gamma W (z') \frac{P^{\mathsf R}_{n} (z')}{z'-z} \d z' ,
\end{align*} 
for $ n \in \mathbb N$.
From the orthogonality conditions~\eqref{eq:ortogonalidadL} 
and~\eqref{eq:ortogonalidadR} we have, for all $n \in \N$, the fol\-low\-ing asymptotic expansion near infinity 
\begin{align*} 
Q^{\mathsf L}_n (z) 
& = - C_n^{-1} \big( I_N z^{-n-1} + q_{\mathsf L,n}^1 z^{-n-2} + \cdots \big), & |z|&\to\infty,
\\
Q^{\mathsf R}_n (z) 
& = - \big( I_N z^{-n-1} + q_{\mathsf R,n}^1 z^{-n-2} + \cdots \big)C_n^{-1} , & |z|&\to\infty . 
\end{align*}


The layout of the paper is as follows. In \S 2 we give a brief introduction 
to Riemann--Hilbert problem for matrix biorthogonal polynomials deriving  the three term recurrence relation, discussing the  Pearson--Laguerre matrix weights with a finite end point and introducing  constant jump fundamental matrix and the important  structure matrix. Then, in \S 3 we give an explicit example of Laguerre matrix weight and in \S 4 we apply these ideas to  differential relations and eigenvalue problems for second order matrix differential operators of Laguerre type. Then, in \S 5 we end the paper with the finding of a matrix extension of an instance of the discrete Painlevé IV equation.

\section{Riemann--Hilbert problem for Matrix Biorthogonal Polynomials} \label{sec:2}

\subsection{The Riemann--Hilbert problem}


We begin this section stating a general theorem on Riemann--Hilbert problem for the Laguerre general weights. A preliminary version of this can be found in~\cite{BFM2}.
\begin{teo} \label{teo:LRHP}
 Given a regular Laguerre type matrix of weights $W(x)$ with support on $\gamma$ we have:
\begin{itemize}
\item[i)] The matrix function 
\begin{align*} 
Y^{\mathsf L}_n (z) & : = 
\begin{bmatrix}
 {P}^{\mathsf L}_{n} (z) & Q^{\mathsf L}_{n} (z) \\[.05cm]
-C_{n-1} {P}^{\mathsf L}_{n-1} (z) & -C_{n-1} Q^{\mathsf L}_{n-1} (z)
\end{bmatrix}
\end{align*}
is, for each $n \in \N$, the unique solution of the Riemann--Hilbert problem, which consists in the determination of a $2N \times 2 N$ complex matrix function such~that:

\hangindent=.9cm \hangafter=1
{\noindent}$\phantom{ol}${\rm (RHL1):} $ Y_n^{\mathsf L} (z)$ is holomorphic in $\C \setminus \gamma $.

\hangindent=.9cm \hangafter=1
{\noindent}$\phantom{ol}${\rm (RHL2):} Has the following asymptotic behavior near infinity,
\begin{align*} 
Y_n^{\mathsf L} (z) = \Big( I_N + \sum_{j=1}^{\infty} (z^{-j})Y_n^{j,\mathsf L} \Big)
\begin{bmatrix}
I_N z^n & {0}_N \\
{0}_N & I_N z^{-n} 
\end{bmatrix} .
\end{align*}

\hangindent=.9cm \hangafter=1
{\noindent}$\phantom{ol}${\rm (RHL3):} Satisfies the jump condition
\begin{align*} 
\big( Y^{\mathsf L}_n (z) \big)_+ &
= \big( Y^{\mathsf L}_n (z) \big)_- \,
\begin{bmatrix}
I_N & W (z) \\ 
{0}_N & I_N 
\end{bmatrix}, &z &\in \gamma .
\end{align*}

\hangindent=.9cm \hangafter=1
{\noindent}$\phantom{ol}${\rm (RHL4):} 
$
Y^{\mathsf L}_n (z) 
= \begin{bmatrix}
\operatorname{O} (1) & s^{\mathsf L}_{1}(z) \\[.05cm]
\operatorname{O} (1) & s^{\mathsf L}_{2}(z) 
\end{bmatrix} $, 
as $ z \to 0$, and
$
\displaystyle 
\lim_{z \to 0} z s^{\mathsf L}_j(z) = 0_N
$, 
$j=1,2$ and the $ \operatorname{O} $ conditions are understood entrywise.

\item[ii)] The matrix function 
\begin{align*}
{Y}^{\mathsf R}_n (z) &
 : = 
\begin{bmatrix}
P^{\mathsf R}_{n} (z) & - P^{\mathsf R}_{n-1} (z) C_{n-1} \\[.05cm]
{Q}^{\mathsf R}_{n} (z) & - {Q}^{\mathsf R}_{n-1} (z) C_{n-1}
\end{bmatrix}
\end{align*}
is, for each $n \in \N$, the unique solution of the Riemann--Hilbert problem, which consists in the determination of a $2N \times 2 N$ complex matrix function such~that:
 
\hangindent=.9cm \hangafter=1
{\noindent}$\phantom{ol}${\rm (RHR1):} $ Y_n^{\mathsf R} (z) $ is holomorphic in $\C \setminus \gamma $.
 
\hangindent=.9cm \hangafter=1
{\noindent}$\phantom{ol}${\rm (RHR2):} Has the following asymptotic behavior near infinity,
\begin{align*} 
Y_n^{\mathsf R} (z) = 
\begin{bmatrix}
I_N z^n & {0}_N \\ 
{0}_N & I_N z^{-n} 
\end{bmatrix}
\Big( I_N + \sum_{j=1}^{\infty} (z^{-j})Y_n^{j,\mathsf R} \Big) .
\end{align*}
 
\hangindent=.9cm \hangafter=1
{\noindent}$\phantom{ol}${\rm (RHR3):} Satisfies the jump condition
\begin{align*} 
\big( Y^{\mathsf R}_n (z) \big)_+ &
=
\begin{bmatrix}
 I_N & {0}_N \\ 
 W (z) & I_N 
\end{bmatrix} \big( Y^{\mathsf R}_n (z) \big)_- , &z &\in \gamma .
\end{align*} 

\hangindent=.9cm \hangafter=1
{\noindent}$\phantom{ol}${\rm (RHR4):} $
Y^{\mathsf R}_n (z) 
= 
\begin{bmatrix}
\operatorname{O} (1) & \operatorname{O} (1) \\ 
s^{\mathsf R}_1(z) & s^{\mathsf R}_2(z) 
\end{bmatrix} $, 
as $ z \to 0$, and
$\displaystyle \lim_{z \to 0} z s^{\mathsf R}_j (z) =0_N$, 
$j=1,2$ and the $ \operatorname{O} $ 
conditions are understood entrywise.

\item[iii)] The determinant of $ Y^{\mathsf L}_n (z) $ and $ Y^{\mathsf R}_n (z) $ are both equal to $1$, for every $ z \in \mathbb{C}$.
\end{itemize}
\end{teo}

\begin{proof}
Using the standard calculations from the scalar case it follows that the matrices $Y^{\mathsf L}_n$ and $Y^{\mathsf R}_n$ satisfy
${\rm (RHL1)}$--${\rm (RHL3)} $ and ${\rm (RHR1)}$--${\rm (RHR3)} $ respectively. 

The entries $W^{j,k}$ of the matrix measure $W$ are given in~\eqref{eq:the_weights}.
It holds (cf.~\cite{gakhov}) that in a neighborhood of the origin the Cauchy transform
\begin{align*}
\phi_{m}(z) = \frac{1}{2 \pi \ii} \int_{\gamma} \frac{ p(\zeta) h_m(\zeta) \zeta^{\alpha_{m}} \log^{p_{m}} \zeta}{\zeta-z} \d \zeta,
\end{align*}
where $p(\zeta)$ denotes any polynomial in $\zeta$,
satisfies
$
\displaystyle
\lim_{z \to 0} z \phi_{m}(z) =0
$.
Then, ${\rm (RHL4)}$ and ${\rm (RHR4)}$ are fulfilled by the matrices $Y^{\mathsf L}_n,Y^{\mathsf R}_n$, respectively.
Now, let us consider the matrix~function
\begin{align*}
G(z) = {Y}^{\mathsf L}_n (z) 
\begin{bmatrix}
0_N&I_N\\[.05cm]
-I_N &0_N
\end{bmatrix}
{Y}^{\mathsf R}_n(z)
\begin{bmatrix}
0_N&-I_N\\[.05cm]
I_N &0_N
\end{bmatrix}
 .
\end{align*}
It can easily be proved that $G(z)$ has no jump or discontinuity on the curve $\gamma$ and that its behavior at the end point $c$ is given by
\begin{align*}
G(z) &
= \begin{bmatrix}
 \operatorname{O}(1) s^{\mathsf L}_1(z) + \operatorname{O} (1)s^{\mathsf R}_2(z) & \operatorname{O}(1) s^{\mathsf L}_1(z) + \operatorname{O}(1) s^{\mathsf R}_1(z) \\[.15cm]
 \operatorname{O} (1)s^{\mathsf L}_2(z) + \operatorname{O} (1) s^{\mathsf R}_2(z) & \operatorname{O} (1) s^{\mathsf L}_2(z) + \operatorname{O} (1) s^{\mathsf R}_1(z)
\end{bmatrix} , & 
z&\to 0,
\end{align*}
so it holds that
$
\displaystyle \lim_{z \to 0} z G(z) = 0
$
and we conclude that the origin is a removable singularity of $G$. 
Now, from the behavior for $ z \to \infty$,
\begin{align*}
G(z) = 
\begin{bmatrix}
I_N z^n & {0}_N \\
{0}_N & I_N z^{-n} 
\end{bmatrix}
\begin{bmatrix}
0_N&I_N\\[.05cm]
-I_N &0_N
\end{bmatrix}
 \begin{bmatrix}
 I_N z^n & {0}_N \\ 
 {0}_N & I_N z^{-n} 
 \end{bmatrix}
\begin{bmatrix}
0_N&-I_N\\[.05cm]
I_N &0_N
\end{bmatrix} = \begin{bmatrix}
I_N&0_N\\[.05cm]
0_N &I_N
\end{bmatrix},
\end{align*}
hence the Liouville theorem implies that $G(z)=I_{2N}$, and the uniqueness of the solution of these Riemann--Hilbert follow.
\end{proof}

Consequences of the previous result and the proof given for it follow.
\begin{coro}
 \label{teo:RRHPinverse}
It holds~that
\begin{align*}
\big( {Y}^{\mathsf L}_n (z)\big)^{-1} = 
\begin{bmatrix}
0_N&I_N\\[.05cm]
-I_N &0_N
\end{bmatrix}
 {Y}^{\mathsf R}_n(z)
\begin{bmatrix}
0_N&-I_N\\[.05cm]
I_N &0_N
\end{bmatrix},
\end{align*}
%
%
%
that entrywise read as follows
\begin{align}
Q_{n}^{\mathsf L} (z) P_{n-1}^{\mathsf R} (z)
 - P_{n}^{\mathsf L} (z) Q_{n-1}^{\mathsf R} (z) & = C_{n-1}^{-1} , 
 \\
P_{n-1}^{\mathsf L} (z) Q_{n}^{\mathsf R} (z)
 - Q_{n-1}^{\mathsf L} (z) P_{n}^{\mathsf R} (z) & = C_{n-1}^{-1} , 
 \\
 Q_{n}^{\mathsf L} (z) P_{n}^{\mathsf R} (z)
 - P_{n}^{\mathsf L} (z) Q_{n}^{\mathsf R} (z) & = 0 
 .
\end{align}
\end{coro}


\subsection{Three term recurrence relation}

Following standard arguments we find
\begin{align*}
Y_{n+1}^{\mathsf L} (z) & 
 = T^{\mathsf L}_n (z) Y^{\mathsf L}_n (z),&
 T^{\mathsf L}_n (z) &
 := 
\begin{bmatrix}
z I_N - \beta_n^{\mathsf L} & C_{n}^{-1} \\[.05cm]
-C_{n} & {0}_N
\end{bmatrix},
& 
n &\in \mathbb N,
\end{align*}
where $ T^{\mathsf L}_n$
denotes 
the left \emph{transfer matrix}.
For the right orthogonality, we similarly obtain,
\begin{align*}
Y_{n+1}^{\mathsf R} (z)&
= Y^{\mathsf R}_n (z) T^{\mathsf R}_n (z) , &
T^{\mathsf R}_n (z)&
:= 
\begin{bmatrix}
z I_N - \beta_n^{\mathsf R} & -C_{n} \\[.05cm]
C_{n}^{-1} & {0}_N
\end{bmatrix}, &  n &\in \mathbb N,
\end{align*}
where $ T^{\mathsf L}_n$
denotes 
the right \emph{transfer matrix}.

Hence, we conclude that the sequence of monic polynomials $\big\{ {P}^{\mathsf L}_n (z)\big\}_{n\in\N} $ satisfies the three term recurrence relations
\begin{align*} 
z {P}^{\mathsf L}_n (z) = {P}^{\mathsf L}_{n+1} (z) + \beta^{\mathsf L}_n {P}^{\mathsf L}_{n} (z) + \gamma^{\mathsf L}_n {P}^{\mathsf L}_{n-1} (z), 
&& n \in \mathbb N, 
\end{align*}
with recursion coefficients given by
$\beta^\mathsf L_n :=
p_{\mathsf L,n}^1 - p_{\mathsf L,n+1}^1$ and $\gamma^{\mathsf L}_n := C_n^{-1} C_{n-1}$,
 with initial conditions, $ {P}^{\mathsf L}_{-1} = 0_N, {P}^{\mathsf L}_{0} = I_N $.
 Analogously, 
 \begin{align*} 
 z P^{\mathsf R}_n (z) = P^{\mathsf R}_{n+1} (z) + P^{\mathsf R}_{n} (z) \beta^{\mathsf R}_n + P^{\mathsf R}_{n-1} (z) \gamma^{\mathsf R}_n ,
 && n \in\N,
\end{align*}
where $\beta_n^{\mathsf R} := C_n \beta^{\mathsf L}_n C_n^{-1}$ and $\gamma^{\mathsf R}_n := C_n \gamma^{\mathsf L}_nC_n^{-1} = C_{n-1} C_n^{-1}$.

\subsection{Pearson--Laguerre matrix weights with a finite end point}

Instead of a given matrix of weights we consider two matrices of entire functions, say $h^{\mathsf L}(z)$ and $h^{\mathsf R}(z)$, such that the following matrix Pearson equations are satisfied
\begin{align}\label{eq:partial_Pearson_L}
z (W^\mathsf L)^{\prime}(z) = h^{\mathsf L}(z) W^{\mathsf L}(z), &&
 z (W^\mathsf R)^{\prime}(z)
 = W^{\mathsf R}(z) h^{\mathsf R}(z) , 
\end{align}
and, given solutions to them, we construct the corresponding matrix of weights as $ W= W^\mathsf L W^\mathsf R$. This matrix of weights is also characterized by a Pearson equation,

\begin{pro}[Pearson differential equation] 
\label{prop:Pearson}
Given two matrices of entire functions $h^{\mathsf L}(z)$ and $h^{\mathsf R}(z)$, any solution of the Sylvester type matrix differential equation, which we call Pearson equation for the~weight,
\begin{align}\label{eq:Pearson}
z W^{\prime}(z)=h^{\mathsf L} (z) W(z)+ W(z) h^{\mathsf R} (z),
\end{align} 
is of the form $ W= W^\mathsf L W^\mathsf R$ where the factor matrices $ W^\mathsf L$ and $ W^\mathsf R$ are solutions of~\eqref{eq:partial_Pearson_L}.
\end{pro}

\begin{proof}
 Given solutions $ W^\mathsf L$ and $ W^\mathsf R$ of~\eqref{eq:partial_Pearson_L}, 
 it follows immediately, just using the Leibniz law for derivatives, that $ W= W^\mathsf L W^\mathsf R$ fulfills~\eqref{eq:Pearson}. Moreover, given a solution $ W$ of~\eqref{eq:Pearson} we pick a solution $ W^\mathsf L$ of the first equation in~\eqref{eq:partial_Pearson_L}, then it is easy to see that $( W^\mathsf L)^{-1} W$ satisfies the second equation in~\eqref{eq:partial_Pearson_L}.
\end{proof}

We can give the fol\-low\-ing result from the literature~\cite{wasow}.
\begin{teo}[Solution at a regular singular point] 
Let the matrix function $h^{\mathsf L}(z) $ be entire. Then, for the solutions of the Pearson equation \eqref{eq:partial_Pearson_L} we have:
\begin{enumerate}
	\item If~$A^{\mathsf L}:=h^{\mathsf L}(0)$ has no eigenvalues that differs from each other by positive integers then, the solution of the left matrix differential equation in~\eqref{eq:partial_Pearson_L}
can be written as 
\begin{align*}
 W^{\mathsf L}(z) = H^{\mathsf L}(z) z^{A^{\mathsf L}} W_0^{\mathsf L},
\end{align*}
where 
$
H^{\mathsf L}(z)$ is an entire and nonsingular matrix function such that
$H^{\mathsf L}(0) =I_N$, and $W_0^{\mathsf L}$ is a constant nonsingular matrix.
 \\

\item If the matrix function
$A^{\mathsf L}$ has eigenvalues that differs from each other by positive integers, then the solution of the left matrix differential equation in~\eqref{eq:partial_Pearson_L} can be written as
\begin{align*}
 W^{\mathsf L}(z) = H^{\mathsf L}(z) z^{\tilde A^{\mathsf L}} W_0^{\mathsf L},
\end{align*}
where, in this case,
\begin{align*}
H^{\mathsf L}(z) = \tilde{S}^{\mathsf L}(z) \Pi^{\mathsf L}(z) ,
\end{align*}
and $ \tilde{S}^{\mathsf L}(z) $ is a finite product of factors of the form $T_i S_i^{\mathsf L}(z)$, with $T_i$
a nonsingular matrix and~$S_i^{\mathsf L}(z)$ 
is a shearing matrix,~{i.e.}, a matrix given by blocks as
\begin{align*}
S_i^{\mathsf L}(z) = 
\begin{bmatrix}
 I_{n_i} & 0\\
 0 & z I_{m_i}
\end{bmatrix} , 
\end{align*}
for some positive integers $n_i, m_i$, and
$\Pi^{\mathsf L}(z)$ is an entire and non singular matrix function such that
$\Pi^{\mathsf L}(0) =I$, $\tilde A^{\mathsf L}$ is a constant matrix built from the matrix $A^{\mathsf L}$, where the eigenvalues of this matrix are decreased in such a way that the eigenvalues of the resulting matrix do not differ by a positive integer and $W_0^{\mathsf L}$ is a constant nonsingular matrix.
\end{enumerate}
\end{teo}

We can get analogous results for the 
right matrix differential equation in~\eqref{eq:partial_Pearson_L}
and we will denote the solution as
\begin{align*}
 W^{\mathsf R}(z) =W_0^{\mathsf R} z^{A^{\mathsf 
 R}} H^{\mathsf R}(z).
\end{align*}
 
The matrix of functions $z^A = \Exp{A \log z}$ is a matrix of holomorphic functions in $\mathbb{C} \setminus \gamma$, and
\begin{align*}
(z^A)_- &= (z^A)_+ \Exp{2 \pi i A} = \Exp{2 \pi i A} (z^A)_+ , & z &\in \gamma.
\end{align*}

We also adopt the convention that $(W^{\mathsf L}(z)W^{\mathsf R}(z) )_+= W(z)$,~i.e., the matrix of weight is obtained from the limit behavior of the right side of the curve $\gamma$ of the matrix function $W^{\mathsf L}(z)W^{\mathsf R}(z)$.

It is necessary, in other to consider the Riemann--Hilbert problem related to the matrix of weights $W$ satisfying~\eqref{eq:Pearson}, to study the behavior of $W(z)$ around the origin. For that aim, let us consider $J$, the Jordan matrix similar to the matrix $A$, so there exists an nonsingular matrix $P$ such that
$A = P J P^{-1}$.
It holds
$z^{A} = P z^{J} P^{-1}$
so if 
\begin{align*}
J = (\lambda_1I_{m_1}+ N_1) \oplus (\lambda_2 I_{m_2}+ N_2) \oplus \cdots \oplus (\lambda_s I_{m_s}+ N_s) 
\end{align*}
where $m_k$ is the order of the nilpotent matrix $N_k$, we have that
\begin{align*}
z^J = z^{ \lambda_1I_{m_1} + N_1 } \oplus z^{\lambda_2 I_{m_2} + N_2 } \oplus \cdots \oplus z^{\lambda_s I_{m_s}+ N_s }
\end{align*}
where
$z^{ \lambda_k I_{m_k} + N_k } = z^{ \lambda_kI_{m_k} } z^{ N_k }$.
It is straightforward that
$z^{ \lambda_k I_{m_k} } = z^{ \lambda_k} I_{m_k} $
and
\begin{align*}
z^{ N_k } = \Exp{ N_k \log z} = I_{m_k} + \log z N_k + \frac{ \log^2 z }{ 2! } N^2_k + \cdots + \frac{\log^{m_k-1} z}{(m_k-1)!}N^{m_k-1}_k ,
\end{align*}
where we have used the nilpotency of $ N^{j}_k = 0_N$ for $ j \geq m_k$, so we can conclude that the entries of $z^{A}$ are linear combinations of $z^{\lambda_j}$ 
with polynomials coefficients in the variable~$\log z$.
Hence, if we assume a real spectrum $\sigma(A)=\{\lambda_1,\ldots,\lambda_s\}\subset \R$ bounded from below by~$-1$, $\lambda_k>-1$, as well as the regularity of the matrix weight $W$, it holds that this matrix of weights is of Laguerre type and fulfills the conditions requested in Theorem~\ref{teo:LRHP}.

%

\subsection{Constant jump fundamental matrix}
According with the above notation and given a matrix of weights as described in~\eqref{eq:Pearson}, with spectra $\sigma(A^{\mathsf L})$ and $\sigma(A^{\mathsf R})$, both real and bounded from below by $-1$, we introduce the \emph{constant jump fundamental matrices} 
\begin{align} \label{eq:zn1} 
Z_n^{\mathsf L}(z) 
& : = Y^{\mathsf L}_n (z) 
\begin{bmatrix} 
W^{\mathsf L} (z) & 0_N \\ 
0_N & ( W^{\mathsf R} (z))^{-1} 
\end{bmatrix} , \\
\label{eq:zetan}
{Z}^{\mathsf R}_n (z)& : =
\begin{bmatrix} 
W^{\mathsf R} (z) & 0_N \\ 
0_N & ( W^{\mathsf L} (z))^{-1} 
\end{bmatrix}
{Y}^{\mathsf R}_n (z) , 
\end{align}
for $n \in \mathbb N$.

\begin{pro} \label{prop:zn}
The constant jump fundamental matrices~$Z^{\mathsf L}_n(z) $ and $Z^\mathsf R_n(z)$ satisfy, for each $n \in \N$, the fol\-low\-ing properties: 

\hangindent=.9cm \hangafter=1
{\noindent}$\phantom{ol}${\rm i)}
Are holomorphic on $\C \setminus \gamma$.

\hangindent=.9cm \hangafter=1
{\noindent}$\phantom{ol}${\rm ii)}
Present the fol\-low\-ing \emph{constant jump condition} on $\gamma$
\begin{align*} 
\big( Z^{\mathsf L}_n (z) \big)_+ &
= \big( Z^{\mathsf L}_n (z) \big)_- 
\begin{bmatrix} 
{(W_0^{\mathsf L}})^{-1}\Exp{- 2 \pi \ii A^{\mathsf L}} W_0^{\mathsf L}& {(W_0^{\mathsf L}})^{-1}\Exp{- 2 \pi \ii A^{\mathsf L}} W_0^{\mathsf L} \\ 
0_N & W_0^{\mathsf R}\Exp{2 \pi i A^{\mathsf R}} ({W_0^{\mathsf R}})^{-1}
\end{bmatrix},\\
\big( {Z}^{\mathsf R}_n (z) \big)_+ &
=
\begin{bmatrix} 
W_0^{\mathsf R}\Exp{-2 \pi i A^{\mathsf R}} ({W_0^{\mathsf R}})^{-1}& {0}_N \\ 
W_0^{\mathsf R}\Exp{-2 \pi i A^{\mathsf R}} ({W_0^{\mathsf R})}^{-1}& {W_0^{\mathsf L}}^{-1}\Exp{2 \pi \ii A^{\mathsf L}} W_0^{\mathsf L}
\end{bmatrix} 
\big( {Z}^{\mathsf R}_n (z) \big)_- ,
\end{align*}
for all $z\in\gamma$.
\end{pro}

\begin{proof}
\hangindent=.9cm \hangafter=1
{\noindent}$\phantom{ol}${\rm i)}
 The holomorphic properties of $Z^{\mathsf L}_n$ are inherit from that of the fundamental matrices $Y_n^{\mathsf L}$ and $z^A$ and taking into account that $H^{\mathsf L}(z)$ is an entire matrix function.

\hangindent=.9cm \hangafter=1
{\noindent}$\phantom{ol}${\rm ii)} 
From the definition of $Z^{\mathsf L}_n (z)$ we have
\begin{align*}
\big(Z^{\mathsf L}_n (z) \big)_+ = \big( Y^{\mathsf L}_n (z) \big)_+ 
\begin{bmatrix}
(W^{\mathsf L} (z))_+ & 0_N \\ 
0_N & ( W^{\mathsf R} (z))_+^{-1} 
\end{bmatrix}
\end{align*}
and taking into account Theorem~\ref{teo:LRHP} we 
successively get
\begin{align*}
\small
\big(Z^{\mathsf L}_n (z) \big)_+ & 
\small 
= \big( Y^{\mathsf L}_n (z) \big)_-
\begin{bmatrix} 
I_N & (W^{\mathsf L} (z) W^{\mathsf R} (z))_+ \\ 
0_N & I_N
\end{bmatrix}
\begin{bmatrix} 
(W^{\mathsf L} (z))_+ & 0_N \\ 
0_N & ( W^{\mathsf R} (z))_+^{-1} 
\end{bmatrix}\\
&
\small = \big( Y^{\mathsf L}_n (z) \big)_-
\begin{bmatrix} 
(W^{\mathsf L} (z))_- & 0_N \\ 
0_N & ( W^{\mathsf R} (z))_-^{-1} 
\end{bmatrix}
\begin{bmatrix} 
(W^{\mathsf L} (z))_-^{-1} & 0_N \\ 
0_N & ( W^{\mathsf R} (z))_- 
\end{bmatrix}
\begin{bmatrix} 
(W^{\mathsf L} (z))_+ & (W^{\mathsf L} (z))_+ \\ 
0_N & ( W^{\mathsf R} (z))_+^{-1} 
\end{bmatrix}\\
&
\small
= 
\big(Z^{\mathsf L}_n (z) \big)_- 
\begin{bmatrix} 
(W^{\mathsf L} (z))_-^{-1} (W^{\mathsf L} (z))_+ &(W^{\mathsf L} (z))_-^{-1} (W^{\mathsf L} (z))_+ \\ 
0_N & W^{\mathsf R} (z)_- ( W^{\mathsf R} (z))_+^{-1} 
\end{bmatrix}
 \\
 &
\small 
=\big(Z^{\mathsf L}_n (z) \big)_-
\begin{bmatrix} 
({W_0^{\mathsf L}})^{-1}\Exp{- 2 \pi \ii A^{\mathsf L}} W_0^{\mathsf L}& ({W_0^{\mathsf L}})^{-1}\Exp{- 2 \pi \ii A^{\mathsf L}} C^{\mathsf L} \\ 
0_N & W_0^{\mathsf R}\Exp{2 \pi \ii A^{\mathsf L}} ({W_0^{\mathsf R}})^{-1}
\end{bmatrix},
\end{align*}
and we get the desired \emph{constant} jump condition for $Z^{\mathsf L}_n (z)$.

To complete the proof we only have to check that
\begin{align}\label{eq:ZLR}
Z^{\mathsf R}_n (z) =
\begin{bmatrix}
0 & -I_N \\ I_N & 0
\end{bmatrix}
(Z_n^{\mathsf L} (z))^{-1}
\begin{bmatrix}
0 & I_N \\ -I_N & 0
\end{bmatrix}.
\end{align}
Which is a consequence of~\eqref{eq:zetan}.
\end{proof}

 \subsection{Structure matrix and zero curvature formula}

In parallel to the matrices $Z^{\mathsf L}_n(z)$ and $Z^{\mathsf R}_n(z)$, for each factorization we introduce what we call \emph{structure matrices} given in terms of the \emph{left}, respectively \emph{right}, logarithmic derivatives by,
\begin{align} \label{eq:Mn}
M^{\mathsf L}_n (z) & 
 :
= \big(Z^{\mathsf L}_n\big)^{\prime} (z) \big(Z^{\mathsf L}_{n} (z)\big)^{-1},&
{M}^{\mathsf R}_n (z) & 
 :
= \big ({Z}^{\mathsf R}_{n} (z)\big)^{-1} \big(Z^{\mathsf R}_n\big)^{\prime} (z) .
\end{align}
It is not difficult to prove that
\begin{align*}
{M}^{\mathsf R}_{n}(z) &
 = -
\begin{bmatrix}
0 & -I_N \\ I_N & 0
\end{bmatrix}
M_n^{\mathsf L} (z)
\begin{bmatrix}
0 & I_N \\ -I_N & 0
\end{bmatrix} , &
n \in \mathbb N .
\end{align*}

\begin{pro}[\cite{BFM}] \label{prop:Mn}
\begin{enumerate}


\item 
The transfer matrices satisfy
 \begin{align*}
 T^{\mathsf L}_n (z)Z_n^{\mathsf L}(z) & = Z^{\mathsf L}_{n+1}(z) ,& {Z}^{\mathsf R}_n(z){T}^{\mathsf R}_n (z) &=
 {Z}^{\mathsf R}_{n+1}(z), & n \in \mathbb N .
 \end{align*}

\item 
The zero curvature formulas 
\begin{align*} 
\begin{bmatrix}
I_N & 0_N\\
0_N & 0_N
\end{bmatrix}
 & 
= M^{\mathsf L}_{n+1} (z) T^{\mathsf L}_n(z) - T^{\mathsf L}_n (z) M^{\mathsf L}_{n} (z),\\
\begin{bmatrix}
I_N & 0_N \\
0_N & 0_N
\end{bmatrix}
 & 
=
T^{\mathsf R}_n (z) \, M^{\mathsf R}_{n+1} (z) 
- M^{\mathsf R}_{n} (z) T^{\mathsf R}_n (z) , 
\end{align*}
$n \in \mathbb N $, are fulfilled.
\end{enumerate}
\end{pro}

Now, we discuss the holomorphic properties of the structure matrices just introdiced.
\begin{teo} \label{prop:mn}
The structure matrices~$M^{\mathsf L}_n(z) $ and $M^\mathsf R_n(z)$ are, for each $n \in \N$ meromorphic on $\C $, with singularities located at $z=0$, which happen to be a removable singularity or a simple pole.
\end{teo}

\begin{proof}
Let us prove the statement for $M^{\mathsf L}_n(z) $, for $M^{\mathsf R}_n(z) $ one should proceed similarly. From~\eqref{eq:Mn} it follows that $M^{\mathsf L}_n(z) $ is holomorphic in $ \C \setminus \gamma$. Due to the fact that $Z^{\mathsf L}_n(z) $
has a constant jump on the curve $\gamma$, the matrix function $\displaystyle \big(Z^{\mathsf L}_n\big)^{\prime} $ has the same constant jump on the curve $\gamma$, so the matrix $M^{\mathsf L}_n(z) $ has no jump on the curve $\gamma$, and it follows that
at the origin $M^{\mathsf L}_n(z) $ has an isolated singularity. 
From~\eqref{eq:Mn} and~\eqref{eq:zn1} it holds
\begin{align*}
M^{\mathsf L}_n (z) = \big(Z^{\mathsf L}_n\big)^{\prime} (z)\big({Z^{\mathsf L}_n} (z) \big)^{-1} = \big(Y^{\mathsf L}_n\big)^{\prime} (z)\big( {Y^{\mathsf L}_n} (z) \big)^{-1} 
+ \frac{1}{z} Y^{\mathsf L}_n (z)
\begin{bmatrix}
h^{\mathsf L}(z) & 0_N\\
0_N & -h^{\mathsf R} (z)
\end{bmatrix}
\big( {Y^{\mathsf L}_n} (z)\big)^{-1}, 
\end{align*}
where
\begin{align*} 
Y^{\mathsf L}_n (z) &  = 
\begin{bmatrix}
 {P}^{\mathsf L}_{n} (z) & Q^{\mathsf L}_{n} (z) \\[.05cm]
-C_{n-1} {P}^{\mathsf L}_{n-1} (z) & -C_{n-1} Q^{\mathsf L}_{n-1} (z)
\end{bmatrix}.
\end{align*}
Each entry of the matrix $Q^{\mathsf L}_{n} (z)$ is a Cauchy transform of certain function $f(z)$, where
$ f(z) = \sum\limits_{ i \in I} \phi_i(z) z^{\alpha_i} \log^{p_i} z$, $\phi_i(z)$ is an entire function, $\alpha_i > -1$, $p_i \in \N$, and~$I$ is a finite set of indices.

It is clear that 
$ 
\lim\limits_{z \to 0} z f(z) = 0 
$. 
Now, see \cite[\S 8.3-8.6]{gakhov} and \cite{Muskhelishvili}, its Cauchy transform 
$
\displaystyle
g(z) = \frac{1}{2\pi \ii }\int_{\gamma} \frac{f(t)}{t - z} \d t
$ 
also satisfies the same property
$
\displaystyle \lim_{z \to 0} z g(z) = 0
$.
We can also see that 
$
\displaystyle
\lim_{z \to 0} z^2 g^{\prime} (z)= 0
$.
Indeed,
\begin{align*}
z g^{\prime} (z)&= 
\int_\gamma \frac{z f(t)}{(t-z)^2} \, \d t = \int_\gamma \frac{(z-t) f(t)}{(t-z)^2} \, \d t+ \int_\gamma \frac{t f(t)}{(t-z)^2} \, \d t,\\
&= -\int_\gamma \frac{f(t)}{t-z} \, \d t - \frac{t f(t)}{t - z} \bigg|_\gamma 
+ \int_\gamma \frac{(tf(t))^{\prime}}{t-z} \, \d t =- \frac{t f(t)}{t - z} \bigg|_\gamma 
+ \int_\gamma \frac{t f^{\prime}(t)}{t-z} \, \d t .
\end{align*}
From the boundary conditions, the first term is zero and we get
\begin{align*}
zg^{\prime} (z)
= 
\int_\gamma \frac{tf^{\prime}(t)}{t-z} \, \d t .
\end{align*}
and from the definition of $f$ we get that $t f^{\prime}(t)$ is a function in the class of $f$, that we denote by $v$ and, consequently,
$z^2 g^{\prime} (z) \underset{z \to 0 }{\to} 0$.
From these considerations it follows, 
\begin{align*}
\big(Y^{\mathsf L}_n\big)^{\prime} (z)
 &
= \begin{bmatrix}
\operatorname{O} (1) & r^{\mathsf L}_1 (z)\\[.1cm]
\operatorname{O} (1) & r^{\mathsf L}_2 (z)\end{bmatrix},
 & 
\big( {Y^{\mathsf L}_n(z)}\big)^{-1} 
&
= \begin{bmatrix}
 r^{\mathsf L}_3(z) & r^{\mathsf L}_4 (z)\\[.1cm]
\operatorname{O} (1) &\operatorname{O} (1) 
\end{bmatrix}, & 
z &
\to 0,
\end{align*}
where 
$
\displaystyle 
\lim_{z \to 0} z^2 r^{\mathsf L}_i(z) = 0_N$, 
for $i=1,2$, and 
$
\displaystyle
\lim_{z \to 0} z r^{\mathsf R}_i(z) = 0_N$,
for $i = 3,4$, so it holds that
\begin{align*}
\lim_{z \to 0} 
z^2 \big(Y^{\mathsf L}_n\big)^{\prime} (z) \big({Y^{\mathsf L}_n}\big)^{-1} 
= \lim_{z \to 0} z^2
\begin{bmatrix}
\operatorname{O} (1) r^{\mathsf L}_1(z) + \operatorname{O} (1) r^{\mathsf L}_3(z)& \operatorname{O} (1) r^{\mathsf L}_1(z) + \operatorname{O} (1) r^{\mathsf L}_4(z) \\[0.1cm]
\operatorname{O} (1) r^{\mathsf L}_2(z) + \operatorname{O} (1) r^{\mathsf L}_3(z)&\operatorname{O} (1) r^{\mathsf L}_2(z) + \operatorname{O} (1) r^{\mathsf L}_4(z)
\end{bmatrix}
=0_{2N}.
\end{align*}
Similar considerations leads us to the result that
\begin{align*}
\lim_{z \to 0} z Y^{\mathsf L}_n(z)
\begin{bmatrix}
h^{\mathsf L}(z) & 0_N\\
0_N & -h^{\mathsf R}(z)
\end{bmatrix}
\big( {Y^{\mathsf L}_n(z)}\big)^{-1} =0_{2N} ,
 &&
\text{so we obtain that}
 && 
\lim_{z\to 0} z^2 M^{\mathsf L}_n(z) =0_{2N} ,
\end{align*}
and 
hence 
the matrix function $M^{\mathsf L}_n(z) $ has at most a simple pole at the point~$z=0$.
\end{proof}

\section{Durán--Grünbaum type Laguerre matrix weights}\label{sec:example}

Motivated by cases considered in the literature \cite{duran20052,duran_1,duran_2,duran_5} we want to include here an example of a Laguerre weight. In this case, we are able to explicitly compute the residue matrix at the simple pole at the origin of the structure matrix.


Let us consider the weight $W (z) = \Exp{A_1 z} z^{\alpha} \Exp{A_2 z}$, $z \in \mathbb C$, defined in $ \mathbb C \setminus [0, + \infty )$ with support on $\gamma=[0,+\infty)$. Here $\alpha,A_1,A_2\in\C^{N\times N}$ are matrices such that $[\alpha,A_1]=[\alpha,A_2]=0_N$, with spectrum 
$\sigma(\alpha) \subset (-1,+\infty)$. 
To match with previous developments in~\cite{duran2004} we just need to shift each of the matrices $A_1$ and $A_2$ by $-I_N$.
Accordingly, we choose
\begin{align*}
W^{\mathsf L} (z) &= \Exp{A_1 z}z^{\frac{\alpha}{2}} , & W^{\mathsf R} (z) = z^{\frac{\alpha}{2}} \Exp{A_2 z}, 
\end{align*}
It can be seen that the matrix function $Z^{\mathsf L}_n$ defined by
\begin{align*}
Z^{\mathsf L}_n (z) = Y^{\mathsf L}_n(z) \mathcal C(z)
, &&  
\text{where} 
 && 
\mathcal C (z)
= \begin{bmatrix}
W^{\mathsf L} (z) & 0 \\ 
0 & ({W^{\mathsf R}}(z))^{-1} 
\end{bmatrix}
 ,
\end{align*}
with $W^{\mathsf L} (z) W^{\mathsf R} (z) = W (z) $,
satisfies
\begin{itemize}
\item $Z^{\mathsf L}_{n} $ is holomorphic in $\C\setminus [0,+\infty)$.
\item $
(Z^{\mathsf L}_{n} (z))_{+}
=(Z^{\mathsf L}_{n} (z))_{-} 
\begin{bmatrix} 
\Exp{-\ii\pi \alpha }& \Exp{-\ii\pi \alpha} \\
 0 & \Exp{\ii\pi \alpha} 
 \end{bmatrix}$ over 
$ (0,+\infty)$.
\end{itemize}

\subsection{$\sigma(\alpha)\subset (-1,+\infty)\setminus \N $}

In this case the constant jump matrix 
$
\begin{bsmallmatrix}
 \Exp{-\ii\pi \alpha } & \Exp{-\ii\pi \alpha} \\
0 & \Exp{\ii\pi \alpha} 
\end{bsmallmatrix}
$ 
can be block diagonalized. For that aim we consider the matrix
\begin{align*}
P = \begin{bmatrix} 
I_N& \Exp{-\ii\pi \alpha } \\ 0 & \Exp{\ii\pi \alpha }- \Exp{-\ii\pi \alpha }
\end{bmatrix},
&&
\text{such that}
&&
\begin{bmatrix} \Exp{-\ii\pi \alpha }& \Exp{-\ii\pi \alpha }\\ 0 & \Exp{\ii\pi \alpha }\end{bmatrix} P = P 
\begin{bmatrix} \Exp{-\ii\pi \alpha }& 0 \\ 0 & \Exp{ \ii\pi \alpha }
\end{bmatrix} .
\end{align*}
So, over the interval $(0,+\infty)$, we have
\begin{align*}
(Z^{\mathsf L}_{n} (z) P)_{+} =(Z^{\mathsf L}_{n} (z)P)_{-} 
\begin{bmatrix} 
\Exp{-\ii\pi \alpha }& 0\\ 0 & \Exp{\ii\pi \alpha }
\end{bmatrix} .
\end{align*}
For $z\in \mathbb{C} \setminus [0, + \infty)$, let us define the matrix
\begin{align}
\label{eq:psi}
\mathcal \psi(z)
:= 
\begin{bmatrix} 
z^{\frac{\alpha}{2}} & 0 \\ 0 & z^{-\frac{\alpha}{2}}
\end{bmatrix} ,
\end{align}
which satisfies, over $ (0,+\infty)$, the fol\-low\-ing jump condition 
\begin{align*}
(\psi(z))_{+} =(\psi(z))_{-} 
\begin{bmatrix} 
\Exp{-\ii\pi \alpha }& 0\\ 0 & \Exp{\ii\pi \alpha}
\end{bmatrix} .
\end{align*}
Consequently, the matrix 
\begin{align*}
F^{\mathsf L}_n (z) :=Z^{\mathsf L}_n (z) P \psi^{-1} (z)
\end{align*}
has no jump in the interval $(0, + \infty)$. The matrix function $F^{\mathsf L}_n$ has an isolated singularity at the origin which, as we will show now, is a removable singularity,~i.e., 
$
\displaystyle
\lim_{z\to 0}z F^{\mathsf L}_n(z)=0_{2N}
$.
From its definition we have that 
\begin{align*}
zF^{\mathsf L}_n(z)& 
= \begin{bmatrix}
\operatorname{O} (z) & zs^{\mathsf L}_{1}(z) \\[.05cm] 
\operatorname{O} (z) & zs^{\mathsf L}_{2}(z) 
\end{bmatrix}
\begin{bmatrix}
\Exp{A_1 z}z^{\frac{\alpha}{2}} & 0_N \\[.05cm] 
0_N & \Exp{-A_2 z}z^{-\frac{\alpha}{2}}
\end{bmatrix} 
\begin{bmatrix} 
I_N& \Exp{-\ii\pi \alpha } \\[.05cm] 
0_N& \Exp{\ii\pi \alpha }- \Exp{-\ii\pi \alpha }
\end{bmatrix}
\begin{bmatrix} z^{-\frac{\alpha}{2}} & 0 \\[.05cm] 
0 & z^{\frac{\alpha}{2}}
\end{bmatrix}
\\
&
= 
\begin{bmatrix}
\operatorname{O} (z) & zs^{\mathsf L}_{1}(z) \\[.05cm] 
\operatorname{O} (z) & zs^{\mathsf L}_{2}(z) 
\end{bmatrix}
\begin{bmatrix}
\Exp{A_1 z} & \Exp{A_1 z} \Exp{-\ii\pi \alpha }z^{\alpha}\\[.05cm]
0_N & \Exp{-A_2 z} (\Exp{\ii\pi \alpha }- \Exp{-\ii\pi \alpha })
\end{bmatrix}, 
& 
z
&
\to 0 ,
\end{align*} 
and as $z s_1^{\mathsf L}$, $z s_2^{\mathsf L}\to 0_N$ 
as $z\to 0$ 
and $\operatorname{O} (z) z^\alpha \to 0_N$, 
as $z\to 0$ 
(because the eigenvalues of $\alpha$ are bounded from below by $-1$) 
we conclude that $z F_n(z)\to 0_{2N}$, for $z\to 0$.
%
%
%
Hence, $F^{\mathsf L}_n(z)$ is a matrix of entire functions.

Now, we want to compute
$
\displaystyle
F^{\mathsf L}_n(0)
 =
\lim_{z\to 0} 
F^{\mathsf L}_n(z)
$. 
Notice that,
\begin{align*} 
F^{\mathsf L}_n(0) = 
\lim_{z\to 0}
Y^{\mathsf L}_n(z) 
\begin{bmatrix}
\Exp{A_1 z} & \Exp{A_1 z} \Exp{-\ii\pi \alpha }z^{\alpha}\\ 0_N & \Exp{-A_2 z} (\Exp{\ii\pi \alpha }- \Exp{-\ii\pi \alpha })
\end{bmatrix}
 ,
\end{align*}
where the limit of each factor inside the limit do not need to exist. 
Given that 
$\sigma(\alpha)\subset (-1,+\infty)\setminus \N $ 
we first separately compute $F^{\mathsf L}_n(0)$ in the cases, when 
$\sigma(\alpha)\subset (0,+\infty)\setminus \{1,2,\ldots\}$ and 
when
$\sigma(\alpha)\subset (-1,0)$, and then we give 
$F^{\mathsf L}_n(0)$
in 
general. 
\paragraph{\textbf{Case} $\sigma(\alpha)\subset (0,+\infty)\setminus \{1,2,\ldots\}$}  When all the eigenvalues of $\alpha$ are strictly positive then each limit exists and
\begin{align*}
F^{\mathsf L}_n(0)
 =Y^{\mathsf L}_n(0) 
\begin{bmatrix}
I_N & 0_N \\ 
0_N & \Exp{\ii\pi \alpha} - \Exp{-\ii\pi \alpha}
\end{bmatrix} .
\end{align*}
\paragraph{\textbf{Case} $\sigma(\alpha)\subset (-1,0)$}  We cannot proceed as before. However, as the limit exists, if we are able to rewrite
\begin{align*}
Y^{\mathsf L}_n(z) 
\begin{bmatrix}
\Exp{A_1 z} & \Exp{A_1 z} \Exp{-\ii\pi \alpha }z^{\alpha} \\ 
0_N & \Exp{-A_2 z} (\Exp{\ii\pi \alpha }- \Exp{-\ii\pi \alpha })
\end{bmatrix}=\hat Y^{\mathsf L}_n(z) f(z) ,
\end{align*}
in terms of two matrix factors $\hat Y^{\mathsf L}_n(z) $ and $f(z)$, a non singular matrix, with $f$ having a well defined limit for $z\to 0$, also being a non-singular matrix, we can ensure that exists 
$
\displaystyle
\lim_{z\to 0} \hat Y^{\mathsf L}_n(z)
$, 
and 
$
\displaystyle
F^{\mathsf L}_n(0) = \big( \lim_{z\to 0}\hat Y^{\mathsf L}_n(z) \big)
\big( \lim_{z\to 0}f(z) \big)
$.
This can be achieved with
\begin{align*}
\hat Y^{\mathsf L}_n(z)
&:= Y^{\mathsf L}_n(z) 
\begin{bmatrix}
z^{-\alpha} &0_N\\
I_N-\Exp{2\ii \pi\alpha} &z^{\alpha}
\end{bmatrix}^{-1}, 
\\
f(z) 
 &:=
\begin{bmatrix}
z^{-\alpha} &0_N\\
I_N-\Exp{2\ii \pi\alpha} &z^{\alpha}
\end{bmatrix}
\begin{bmatrix}
\Exp{A_1 z} & \Exp{A_1 z} \Exp{-\ii\pi \alpha }z^{\alpha}\\ 0_N & \Exp{-A_2 z} (\Exp{\ii\pi \alpha }- \Exp{-\ii\pi \alpha })
\end{bmatrix} \\&
=
\begin{bmatrix}
z^{-\alpha}\Exp{A_1 z} & \Exp{A_1 z} \Exp{-\ii\pi \alpha }\\ (I_N-\Exp{2\ii \pi\alpha})\Exp{A_1 z}&( -\Exp{A_1 z}+ \Exp{-A_2 z}) (\Exp{\ii\pi \alpha }- \Exp{-\ii\pi \alpha })z^\alpha
\end{bmatrix}
\end{align*}
So that, 
\begin{align*}
\lim_{z \to 0}f(z) &
=\begin{bmatrix}
0_N& \Exp{-\ii\pi \alpha }\\ I_N-\Exp{2\ii \pi\alpha}&0_N
\end{bmatrix}, 
 &
F_n^{\mathsf L}(0)&
=\hat Y^{\mathsf L}_n(0)\begin{bmatrix}
0_N& \Exp{-\ii\pi \alpha }\\ I_N-\Exp{2\ii \pi\alpha}&0_N
\end{bmatrix}.
\end{align*}
\paragraph{\textbf{General case} $\sigma(\alpha)\subset (-1,+\infty)\setminus \N $}, 
Recalling the canonical Jordan form, we can write 
$\alpha=PJ P^{-1}$ with
\begin{align*}
J=
\begin{bmatrix}
J^+& 0_{N^+\times N^-} \\
0_{N^-\times N^+} & J^-
\end{bmatrix},
\end{align*}
and $N^+$ ($N^-$) being the sum of the algebraic multiplicities associated with positive (negative) eigenvalues and in $J^+$ ($J^-$) we gather together the Jordan blocks of all positive (negative) eigenvalues. We have
\begin{align*}
\begin{bmatrix}
\Exp{A_1 z} & \Exp{A_1 z} \Exp{-\ii\pi \alpha }z^{\alpha}\\ 0_N & \Exp{-A_2 z} (\Exp{\ii\pi \alpha }- \Exp{-\ii\pi \alpha })
\end{bmatrix}
=
\begin{bmatrix}
P & 0_N\\
0_N &P
\end{bmatrix}
\begin{bmatrix}
\Exp{  A_1 z} & \Exp{\tilde A_1 z} \Exp{-\ii\pi J }z^{J}\\ 0_N & \Exp{-\tilde A_2 z} (\Exp{\ii\pi J }- \Exp{-\ii\pi J })
\end{bmatrix}
\begin{bmatrix}
P & 0_N \\
0_N &P
\end{bmatrix}^{-1}
\end{align*}
with 
$\tilde A_j=P^{-1} A_j P$, $j=1,2$. 
Now, as we did in the previous case, with negative eigenvalues only, we left multiply by the fol\-low\-ing nonsingular matrix
\begin{align*}
S(z):=\begin{bmatrix}
P & 0_N\\
0_N &P
\end{bmatrix}
\begin{bmatrix}
\begin{bmatrix}
I_{N^+} & 0_{N^+\times N^-}\\
 0_{N^-\times N^+}& z^{- J^-}
\end{bmatrix} & 0_N\\
\begin{bmatrix}
0_{N^+} & 0_{N^+\times N^-}\\
 0_{N^-\times N^+}& I_{N^-}- \Exp{2\ii\pi J^-}
\end{bmatrix} & \begin{bmatrix}
I_{N^+} & 0_{N^+\times N^-}\\
0_{N^-\times N^+}& z^{ J^-}
\end{bmatrix}
\end{bmatrix}
\begin{bmatrix}
P & 0_N\\
0_N &P
\end{bmatrix}^{-1}, 
\end{align*}
to get
%
%
%
\begin{align*}
\begin{bsmallmatrix}
P & 0_N\\
0_N &P
\end{bsmallmatrix}
\begin{bsmallmatrix}
\begin{bsmallmatrix}
I_{N^+} & 0_{N^+\times N^-}\\
0_{N^-\times N^+}& z^{- J^-}
\end{bsmallmatrix}
\Exp{\tilde A_1 z}&
\begin{bsmallmatrix}
I_{N^+} & 0_{N^+\times N^-}\\
0_{N^-\times N^+}& z^{- J^-}
\end{bsmallmatrix}
\Exp{\tilde A_1 z}
\begin{bsmallmatrix}
\Exp{-\ii\pi J^+ }z^{J^+ }&0_{N^+\times N^-}\\
0_{N^-\times N^+} & \Exp{-\ii\pi J^- }z^{J^- }
\end{bsmallmatrix}
\\
\begin{bsmallmatrix}
0_{N^+} & 0_{N^+\times N^-}\\
0_{N^-\times N^+}& I_{N^-}- \Exp{2\ii\pi J^-}
\end{bsmallmatrix}
\Exp{\tilde A_1 z} &
\begin{smallmatrix}
\begin{bsmallmatrix}
0_{N^+} & 0_{N^+\times N^-}\\
0_{N^-\times N^+}& I_{N^-}- \Exp{2\ii\pi J^-}
\end{bsmallmatrix}
\Exp{\tilde A_1 z}
\begin{bsmallmatrix}
\Exp{-\ii\pi J^+ }z^{J^+ }&0_{N^+\times N^-}\\
0_{N^-\times N^+} & \Exp{-\ii\pi J^- }z^{J^- }
\end{bsmallmatrix}
\\
+\begin{bsmallmatrix}
I_{N^+} & 0_{N^+\times N^-}\\
0_{N^-\times N^+}& z^{ J^-}
\end{bsmallmatrix}
\Exp{-\tilde A_2 z}
\begin{bsmallmatrix}
\Exp{\ii\pi J^+ }- \Exp{-\ii\pi J^+}& 0_{N^+\times N^-}\\ 0_{N^-\times N^+}& \Exp{\ii\pi J^- }- \Exp{-\ii\pi J_ -}
\end{bsmallmatrix}
\end{smallmatrix}
\end{bsmallmatrix}
\begin{bsmallmatrix}
P & 0_N\\
0_N &P
\end{bsmallmatrix}
^{-1}
\end{align*}
which for $z\to 0$ has a well defined limit, being a non-singular matrix, given by
\begin{align*}
\begin{bmatrix}
P & 0_N \\
0_N & P
\end{bmatrix}
\begin{bmatrix}
\begin{bmatrix}
I_{N^+} & 0_{N^+\times N^-}\\
0_{N^-\times N^+}&0_{N^-}
\end{bmatrix} & 
\begin{bmatrix}
0_{N^+}&0_{N^+\times N^-}\\
0_{N^-\times N^+} & \Exp{-\ii\pi J^- }
\end{bmatrix} \\
\begin{bmatrix}
0_{N^+} & 0_{N^+\times N^-}\\
0_{N^-\times N^+}& I_{N^-}- \Exp{2\ii\pi J^-}
\end{bmatrix} &
\begin{bmatrix}
\Exp{\ii\pi J^+ } - \Exp{-\ii\pi J^+} & 
0_{N^+\times N^-} \\ 
0_{N^-\times N^+}& 0_{N^-}
\end{bmatrix}
\end{bmatrix}
\begin{bmatrix}
P & 0_N\\
0_N &P
\end{bmatrix}^{-1}.
\end{align*}
Thus,
\begin{align*}
\small
F_n^{\mathsf L}(0)
=
\hat Y^{\mathsf L}_n(0)
\begin{bmatrix}
P & 0_N \\
0_N &P
\end{bmatrix}
\begin{bmatrix}
\begin{bmatrix}
I_{N^+} & 0_{N^+\times N^-}\\
0_{N^-\times N^+}&0_{N^-}
\end{bmatrix}& 
\begin{bmatrix}
0_{N^+}&0_{N^+\times N^-}\\
0_{N^-\times N^+} & \Exp{-\ii\pi J^- }
\end{bmatrix} \\
\begin{bmatrix}
0_{N^+} & 0_{N^+\times N^-}\\
0_{N^-\times N^+}& I_{N^-}- \Exp{2\ii\pi J^-}
\end{bmatrix} &
\begin{bmatrix}
\Exp{\ii\pi J^+ }- \Exp{-\ii\pi J^+}& 0_{N^+\times N^-}\\ 0_{N^-\times N^+}& 0_{N^-}
\end{bmatrix}
\end{bmatrix}
\begin{bmatrix}
P & 0_N\\
0_N &P
\end{bmatrix}^{-1} .
\end{align*}
Given
\begin{align*}
M^{\mathsf L}_n =\big( Z^{\mathsf L}_n \big)^{\prime} \big(Z^{\mathsf L}_n \big)^{-1} = 
\big( F^{\mathsf L}_n \big)^{\prime} \big(F^{\mathsf L}_n \big)^{-1} +F^{\mathsf L}_n
\psi^{\prime} \psi ^{-1}\big(F^{\mathsf L}_n \big)^{-1} ,
\end{align*}
as $\det F^{\mathsf L}_n(z) \not =0$, we know that 
$\big( F^{\mathsf L}_n \big)^{\prime} \big(F^{\mathsf L}_n \big)^{-1}$
has no singularities, while
\begin{align*}
F^{\mathsf L}_n \psi^{\prime} \psi^{-1}\big( F^{\mathsf L}_n 
\big)^{-1} = \frac{1}{z} F^{\mathsf L}_n
\begin{bmatrix} 
\frac{\alpha}{2}
&
0_N 
\\ 
0_N
&
- \frac{\alpha}{2} 
\end{bmatrix} 
\big( F^{\mathsf L}_n \big)^{-1} .
\end{align*}
Consequently, $ M^{\mathsf L}_n (z)$ has a simple pole at the origin with
\begin{align*}
M^{\mathsf L}_n (z) &
=
\frac{1}{z}
F^{\mathsf L}_n(0) 
\begin{bmatrix} 
\frac{\alpha}{2} & 0_N \\ 0_N & -\frac{\alpha}{2} 
\end{bmatrix}
\big(F^{\mathsf L}_n (0)\big)^{-1}+\operatorname{O}(1) , & z&\to 0.
\end{align*}

\subsection{$\alpha =m I_N,  m \in \N$}
It can be seen that the matrix function $Z^{\mathsf L}_n$
satisfies over $(0,+\infty)$ the following jump condition
\begin{align*}
(Z^{\mathsf L}_{n} (z))_{+}
=  (Z^{\mathsf L}_{n} (z))_{-} 
 \begin{bmatrix} 
(-1)^m I_N& (-1)^m I_N\\
 0 & (-1)^m I_N 
 \end{bmatrix}
  .
 \end{align*} 
 
For $z\in \mathbb{C} \setminus [0, + \infty)$, instead of $\eqref{eq:psi}$, let us define the matrix
\begin{align*}
\mathcal \psi(z)
:= 
\begin{bmatrix} 
z^{\frac{m}{2}} I_N& -\frac{1}{2\pi i} z^{\frac{m}{2}}\log z I_N\\ 0 &z^{-\frac{m}{2}}I_N
\end{bmatrix},
\end{align*}
where $\log z$ is the branch of the logarithmic function defined in $\C \setminus [0, +\infty)$, which satisfies, over $ (0,+\infty)$, the same jump condition 
\begin{align*}
(\psi(z))_{+} =(\psi(z))_{-} 
 \begin{bmatrix} 
(-1)^m I_N& (-1)^m I_N\\
 0 & (-1)^m I_N 
 \end{bmatrix} .
\end{align*}
Consequently, the matrix 
\begin{align*}
F^{\mathsf L}_n (z) :=Z^{\mathsf L}_n (z) \psi^{-1} (z)
\end{align*}
has no jump in the interval $(0, + \infty)$. The matrix function $F^{\mathsf L}_n$ has an isolated singularity at the origin which, as we will show now, is a removable singularity,~i.e., 
\begin{align*}
zF^{\mathsf L}_n(z)& 
= \begin{bmatrix}
\operatorname{O} (z) & zs^{\mathsf L}_{1}(z) \\[.05cm] 
\operatorname{O} (z) & zs^{\mathsf L}_{2}(z) 
\end{bmatrix}
\begin{bmatrix}
\operatorname{O} (1) & 0_N \\[.05cm] 
0_N &\operatorname{O} (1)
\end{bmatrix} 
 \begin{bmatrix}
\operatorname{O} (1) & \operatorname{O} (\log z)\\[.1cm]
\operatorname{O} (1) &\operatorname{O} (1) \end{bmatrix} \\
&
= 
\begin{bmatrix}
\operatorname{O} (z) + zs^{\mathsf L}_{1}(z)  & \operatorname{O} (z \log z) + zs^{\mathsf L}_{1}(z)  \\[.05cm] 
\operatorname{O} (z)+ zs^{\mathsf L}_{2}(z)  &  \operatorname{O} (z \log z) +zs^{\mathsf L}_{2}(z) 
\end{bmatrix}, 
& 
z
&
\to 0 ,
\end{align*} 
and as $z s_1^{\mathsf L}$, $z s_2^{\mathsf L}\to 0_N$ 
as $z\to 0$, we conclude that $z F_n(z)\to 0_{2N}$, for $z\to 0$.

Hence, $F^{\mathsf L}_n(z)$ is a matrix of entire functions. To compute $F_n(0)$
notice that,
\begin{align*} 
F^{\mathsf L}_n(0) = 
\lim_{z\to 0}
Y^{\mathsf L}_n(z) 
\begin{bmatrix}
\Exp{A_1 z} &\frac{1}{2 \pi i} z^{m}\log z  \Exp{A_1 z} \\ 0_N & \Exp{-A_2 z}  \end{bmatrix}
 .
\end{align*}
For $ m = 1, 2 \dots$ it holds that
$F^{\mathsf L}_n(0) = 
Y^{\mathsf L}_n(0)$.
For $m=0$ the limit of each factor inside the limit do not need to exist.
As the limit exists, let us write
\begin{align*}
Y^{\mathsf L}_n(z) 
\begin{bmatrix}
\Exp{A_1 z} &\frac{1}{2 \pi i} \log z  \Exp{A_1 z} \\ 0_N & \Exp{-A_2 z}  \end{bmatrix}=\hat Y^{\mathsf L}_n(z) f(z) ,
\end{align*}
with
\begin{align*}
\small \hat Y^{\mathsf L}_n(z)
&
\small := Y^{\mathsf L}_n(z) 
\begin{bmatrix}
(\log z)^{-1}I_N&0_N\\
-2 \pi \ii I_N &\log z I_N
\end{bmatrix}^{-1}, 
&
\small f(z) 
 &
\small \begin{aligned}[t]&:=
\begin{bmatrix}
(\log z)^{-1}I_N&0_N\\
-2 \pi \ii I_N &\log z I_N
\end{bmatrix}
\begin{bmatrix}
\Exp{A_1 z} &\frac{1}{2 \pi \ii} \log z  \Exp{A_1 z} \\ 0_N & \Exp{-A_2 z}  \end{bmatrix}\\
&
\small
=
\begin{bmatrix}
(\log z)^{-1}\Exp{A_1 z} &\frac{1}{2 \pi \ii}  \Exp{A_1 z} \\ -2 \pi \ii \Exp{A_1 z}& -\log z ( \Exp{A_1 z} - \Exp{-A_2 z})
\end{bmatrix}.
 \end{aligned}
\end{align*}
So that, 
\begin{align*}
\lim_{z \to 0}f(z) &
=\begin{bmatrix}
0_N &\frac{1}{2 \pi \ii}  I_N \\ -2 \pi \ii I_N&0_N
\end{bmatrix},
 &
F_n^{\mathsf L}(0)&
=\hat Y^{\mathsf L}_n(0)\begin{bmatrix}
0_N &\frac{1}{2 \pi i}  I_N \\ -2 \pi \ii I_N&0_N
\end{bmatrix}.
\end{align*}

Using the same kind of reasoning as above we get that, $ M^{\mathsf L}_n (z)$ has a simple pole at the origin with
\begin{align*}
M^{\mathsf L}_n (z) &
=
\frac{1}{z}
F^{\mathsf L}_n(0) 
\begin{bmatrix} 
\frac{m}{2} I_N & -\frac{z^m}{2 \pi i}  I_N\\
0_N & -\frac{m}{2}  I_N
\end{bmatrix}
\big(F^{\mathsf L}_n (0)\big)^{-1}+\operatorname{O}(1) , & z&\to 0.
\end{align*}

\section{Eigenvalue problems}

\subsection{Differential relations from the Riemann--Hilbert problem}
We are interested in the differential equations fulfilled by  the biorthogonal matrix polynomials determined by Laguerre type matrices of weights. Different attempts appear in the literature when one considers matrix orthogonality. Here we use the Riemann--Hilbert problem approach in order to derive these differential relations.

We use the notation for the structure matrices
\begin{align*}
\tilde{M}^\mathsf L_n (z)& = z{M}^\mathsf L_n (z) , &\tilde{M}^\mathsf R_n (z)& = z{M}^\mathsf R_n (z) ,
\end{align*}
with $\tilde{M}^\mathsf L_n (z)$ and $\tilde{M}^\mathsf R_n (z)$ matrices of entire functions.


\begin{pro}[First order differential equation for the fundamental matrices $Y^{\mathsf L}_n(z)$ and $Y^{\mathsf R}_n(z)$] \label{prop:mn}
It holds that 
\begin{align}
\label{eq:firstodeYnL}
z \big(Y^{\mathsf L}_n\big)^{\prime} (z) + Y^{\mathsf L}_n (z)
\begin{bmatrix}
h^{\mathsf L} (z) & 0_N \\ 0_N & - h^{\mathsf R} (z)
\end{bmatrix} 
&
= \tilde{M}^{\mathsf L}_n (z) Y^{\mathsf L}_n (z) \\
\label{eq:firstodeYnR}
z \big(Y^{\mathsf R}_n\big)^{\prime} (z)
 +
\begin{bmatrix}
h^{\mathsf R} (z) & 0_N \\ 0_N & - h^{\mathsf L} (z)
\end{bmatrix} 
Y^{\mathsf R}_n (z)
&
= Y^{\mathsf R}_n (z) \tilde{M}^{\mathsf R}_n (z) .
\end{align}
\end{pro}

\begin{proof}
Equations~\eqref{eq:firstodeYnL} and~\eqref{eq:firstodeYnR} follows immediately from the definition of the matrices $M^{\mathsf L}_n(z)$ and $M^{\mathsf R}_n(z)$ in~\eqref{eq:Mn}. 
\end{proof}


We introduce the $\mathcal N $ transform, 
$
\displaystyle 
\mathcal N (F(z)) = F^{\prime}(z)+\frac{F^2(z)}{z}
$.

\begin{pro}[Second order differential equation for the fundamental matrices] \label{prop:mn}
It holds that
\begin{align} \label{eq:edo1}
z \big(Y^\mathsf L_n\big)^{\prime\prime} 
+ \big(Y^\mathsf L_n\big)^{\prime} 
\begin{bmatrix} 
 2 h^{\mathsf L}+ I_N & 0_N \\ 
 0_N & - 2 h^{\mathsf R} + I_N
 \end{bmatrix}
+ Y_n^\mathsf L (z)
 \begin{bmatrix}
 \mathcal N (h^{\mathsf L}) & 0_N \\ 
 0_N &\mathcal N (- h^{\mathsf R}) 
 \end{bmatrix} 
&= \mathcal N ( \tilde{M}^\mathsf L_n )Y^\mathsf L_n ,
 \\
\label{eq:edo2}
z \big(Y^\mathsf R_n\big)^{\prime\prime} 
+
\begin{bmatrix} 
2 h^{\mathsf R} + I_N & 0_N\\ 
 0_N & - 2 h^{\mathsf L} + I_N 
 \end{bmatrix} 
 \big(Y^\mathsf R_n\big)^{\prime} 
 + 
\begin{bmatrix} 
 \mathcal N (h^{\mathsf R}) & 0_N\\ 
 0_N & \mathcal N (-h^{\mathsf L}) 
 \end{bmatrix} 
Y_n{^\mathsf R} (z)
 &= 
Y^\mathsf R_n \mathcal N (\tilde{M}^\mathsf R_n ).
 \end{align}
\end{pro}

\begin{proof}
Differentiating in \eqref{eq:Mn} we get
\begin{align*}
 \big(Z^{\mathsf L}_n \big)^{\prime \prime} \big(Z^{\mathsf L}_{n}\big)^{-1} = \frac{ \big(\tilde{M}^\mathsf L_n \big)^{\prime} }{z} -\frac{ \tilde{M}^\mathsf L_n }{z^2} 
 +
\frac{ (\tilde{M}^\mathsf L_n )^2}{z^2} ,
\end{align*}
so that
\begin{align*}
z \big(Z^{\mathsf L}_n \big)^{\prime \prime} \big(Z^{\mathsf L}_{n}\big)^{-1} + \big(Z^{\mathsf L}_n \big)^{\prime} \big(Z^{\mathsf L}_{n}\big)^{-1} 
= \big(\tilde{M}^\mathsf L_n \big)^{\prime} 
 +
\frac{ (\tilde{M}^\mathsf L_n )^2 }{z} .
\end{align*}
Now, using~\eqref{eq:zn1} and~\eqref{eq:partial_Pearson_L}, we get the stated result~\eqref{eq:edo1}. The equation~\eqref{eq:edo2} follows in a similar way from definition of $M_n^{\mathsf R}$ in~\eqref{eq:Mn}.
\end{proof}

We introduce the fol\-low\-ing $\C^{2N\times 2N}$ valued functions
\begin{align*}
\mathsf H_n^{\mathsf L} =
\begin{bmatrix}
 \mathsf H_{1,1,n}^{\mathsf L} & \mathsf H_{1,2,n}^{\mathsf L} \\[.05cm] 
 \mathsf H_{2,1,n}^{\mathsf L} & \mathsf H_{2,2,n}^{\mathsf L} 
\end{bmatrix}
 &
 : = 
 \mathcal N (\tilde{M}^\mathsf L_n ), &
\mathsf H_n^{\mathsf R} =
\begin{bmatrix}
\mathsf H_{1,1,n}^{\mathsf R} & \mathsf H_{1,2,n}^{\mathsf R} \\[.05cm]
 \mathsf H_{2,1,n}^{\mathsf R} & \mathsf H_{2,2,n}^{\mathsf R} 
\end{bmatrix}
&
:= \mathcal N (\tilde{M}^\mathsf R_n ).
\end{align*}
It holds~that the second order matrix differential equations~\eqref{eq:edo1} and~\eqref{eq:edo2} split in the fol\-low\-ing differential relations
\begin{align}
\label{eq:secondorderpl}
z \big(P_n^{\mathsf L} \big)^{\prime \prime} + \big( P_n^{\mathsf L} \big)^{ \prime}\big( 2 h^{\mathsf L} + I_N \big)
+
P_n^{\mathsf L} \mathcal N ( h^{\mathsf L}) 
& = \mathsf H_{1,1,n}^{\mathsf L} P_n^{\mathsf L} - \mathsf H_{1,2,n}^{\mathsf L} C_{n-1}P_{n-1}^{\mathsf L} , 
 \\
\label{eq:secondorderql}
z \big(Q_n^{\mathsf L} \big)^{\prime \prime}+\big(Q_n^{\mathsf L} \big)^{ \prime}\big( -2 h^{\mathsf R} + I_N \big)
+
Q_n^{\mathsf L} \mathcal N ( -h^{\mathsf R} ) & =
 \mathsf H_{1,1,n}^{\mathsf L} Q_n^{\mathsf L} - \mathsf H_{1,2,n}^{\mathsf L} C_{n-1} Q_{n-1}^{\mathsf L} ,
 \\
\label{eq:secondorderpr}
z \big(P_n^{\mathsf R} \big)^{\prime \prime} + \big( 2 h^{\mathsf R} + I_N \big) \big(P_n^{\mathsf R} \big)^{ \prime}
+ \mathcal N ( h^{\mathsf R} ) P_n^{\mathsf R} 
 &= 
P_n^{\mathsf R} \mathsf H_{1,1,n}^{\mathsf R} 
- P_{n-1}^{\mathsf R} C_{n-1}\mathsf H_{2,1,n}^{\mathsf R} ,\\
\label{eq:secondorderqr}
z \big(Q_n^{\mathsf R} \big)^{\prime \prime}+ \big( - 2 h^{\mathsf L} + I_N \big) \big(Q_n^{\mathsf R} \big)^{ \prime}
+ \mathcal N (- h^{\mathsf L}) Q_n^{\mathsf R} 
 &= 
Q_n^{\mathsf R} \mathsf H_{1,1,n}^{\mathsf R} 
- Q_{n-1}^{\mathsf R} C_{n-1}\mathsf H_{2,1,n}^{\mathsf R}.
\end{align}
Now, we illustrate these constructions with the example discussed in \S~\ref{sec:example}. 
Using the identities $p^1_{R,n} = -q^1_{L,n-1}$ and $p^1_{L,n} = -q^1_{R,n-1}$ we can get 
\begin{gather*}
 M^{\mathsf L}_n (z)= \frac{\tilde{M}^{\mathsf L}_n (z)}{z}
 = \frac{1}{z}
 \begin{bmatrix} 
A_1 z +[ p^1_{L,n}, A_1 ] + n I_N + \frac{\alpha}{2} & A_1 C_n^{-1} + C_n^{-1} A_2 \\[.15cm]
-C_{n-1} A_1 -A_2 C_{n-1} & -A_2 z + [ p^1_{R,n} ,A_2 ] - n I_N- \frac{\alpha}{2} 
 \end{bmatrix} .
\end{gather*}
Using \eqref{eq:edo1}, we obtain the second order differential equation
\begin{gather*}
\begin{multlined}[t][0.75\textwidth]
z \big(Y^{\mathsf L}_n \big)^{\prime \prime} + \big(Y^{\mathsf L}_n \big)^{\prime} \begin{bmatrix} 
 \alpha + I_N + 2 A_1 z & 0_N \\
0_N & I_N-\alpha -2 A_2 z
 \end{bmatrix} 
\\ + Y^{\mathsf L}_n 
 \begin{bmatrix} 
 A_1 +\frac{1}{2} A_1 \alpha + \frac{1}{2} \alpha A_1 + z {A_1}^2 & 0_N\\
0_N & -A_2 +\frac{1}{2} A_2 \alpha + \frac{1}{2} \alpha A_2 + z {A_2}^2
 \end{bmatrix} + \frac{1}{z} Y^{\mathsf L}_n \begin{bmatrix} 
 (\frac{\alpha}{2})^2 & 0_N \\
0_N & (\frac{\alpha}{2})^2 
 \end{bmatrix} 
\end{multlined}
 \\
= \Big( (\tilde{M}^{\mathsf L}_n )^{\prime} (0) + (\tilde{M}^{\mathsf L}_n (0) )^2\frac{1}{z} + (\tilde{M}^{\mathsf L}_n)^{\prime} (0) \tilde{M}^{\mathsf L}_n (0)+ \tilde{M}^{\mathsf L}_n (0) (\tilde{M}^{\mathsf L}_n)^{\prime} (0) + \big( (\tilde{M}^{\mathsf L}_n )^{\prime} (0) \big)^2 z \Big) Y^{\mathsf L}_n (z)
\end{gather*}

As we have proven in \S~\ref{sec:example} for Durán--Grünbaum Laguerre type matrices of weights, under the restriction $[\alpha,A_1]=[\alpha,A_2]=0_N$, and the spectrum of $\alpha$ is contained on $(-1,+\infty)\setminus \mathbb N$ the matrix
$ {M}^{\mathsf L}_n = \big(Z^{\mathsf L}_n \big)^{\prime} \big(Z^{\mathsf L}_n \big)^{-1} $ 
has a pole of order $1$ at $z=0$, with residue given~by
\begin{align*}
 \tilde{M}^{\mathsf L}_n (0) = F^{\mathsf L}_n(0) 
\begin{bmatrix} 
\frac{\alpha}{2} &0_N \\ 0_N &- \frac{\alpha}{2}
\end{bmatrix} 
\big( F^{\mathsf L}_n (0)\big)^{-1} .
\end{align*}
If we now also assume on the matrix $\alpha$ that
 $\alpha^2 = \lambda I_N$,  we get
 \begin{align*}
 (\tilde{M}^{\mathsf L}_n (0))^2 = F^{\mathsf L}_n(0) 
\begin{bmatrix} 
\big( \frac{\alpha}{2} \big)^2 &0_N \\ 0_N &
\big( \frac{\alpha}{2} \big)^2 
\end{bmatrix} 
\big( F^{\mathsf L}_n (0)\big)^{-1} = \frac{\lambda}{4} I_{2N}.
\end{align*}
We remark that as the spectrum of $\alpha$ is contained in $(-1,+\infty)\setminus \mathbb N$ when $|\lambda|<1$  the $\pm \lambda$ are admissible eigenvalues for $\alpha$, and when $|\lambda|>1$ only positive and bigger than $1$ eigenvalues are admissible for $\alpha$, and then $\alpha = \lambda I_N$.
In an analogous way we obtain for $\alpha = m I_N$, $m \in \N 
$
 \begin{align*}
 (\tilde{M}^{\mathsf L}_n (0))^2 = F^{\mathsf L}_n(0) 
\begin{bmatrix} 
\big( \frac{m}{2} \big)^2 &0_N \\ 0_N &
\big( \frac{m}{2} \big)^2 
\end{bmatrix} 
\big( F^{\mathsf L}_n (0)\big)^{-1} = \frac{m^2}{4} I_{2N}.
\end{align*}

In both cases the second order differential equation is simplified to
\begin{multline*}
\phantom{o}\hspace{-.5cm}z \big(Y^{\mathsf L}_n \big)^{\prime \prime} + \big(Y^{\mathsf L}_n \big)^{\prime} \begin{bmatrix} 
 \alpha + I_N + 2 A_1 z & 0_N \\
0_N & I_N-\alpha -2 A_2 z
 \end{bmatrix} 
+ Y^{\mathsf L}_n 
 \begin{bmatrix} 
 A_1 + A_1 \alpha  + {A_1}^2 z & 0_N \\
0_N & -A_2 + A_2 \alpha  +  {A_2}^2 z
 \end{bmatrix} 
 \\
= \begin{bmatrix} 
A_1 +[p^1_{L,n},A_1^2] + ( nI_N + \alpha )A_1 + A_1^2 z & 
A_1^2 C_n^{-1} - C_n^{-1} A_2^2 
 \\
- C_{n-1}A_1^2+ A_2^2C_{n-1} & 
-A_2 -[p^1_{R,n},A_2^2] + ( nI_N + \alpha  ) A_2+ A_2^2 z 
\end{bmatrix} Y^{\mathsf L}_n (z) .
\end{multline*}
Notice that this equation has no pole at zero as it happens in the scalar Laguerre case.
In fact, in the scalar case
this equation 
reduces to
\begin{multline*}
z \big(Y^{\mathsf L}_n \big)^{\prime \prime} + \big(Y^{\mathsf L}_n \big)^{\prime} \begin{bmatrix} 
 \alpha + 1 + 2 A_1 z & 0 \\
0 & 1 -\alpha -2 A_1 z
 \end{bmatrix} 
+ Y^{\mathsf L}_n 
 \begin{bmatrix} 
 A_1 + A_1 \alpha  + {A_1}^2 z & 0 \\
0 & -A_1 + A_1 \alpha  +  {A_1}^2 z
 \end{bmatrix} 
 \\
= \begin{bmatrix} 
A_1  + ( n + \alpha )A_1 + A_1^2 z & 
 0 
 \\
 0 & 
-A_1 + ( n + \alpha  ) A_1+ A_1^2 z 
\end{bmatrix} Y^{\mathsf L}_n (z) ,
\end{multline*} 
as $A_1^2 C_n^{-1} = C_n^{-1} A_1^2 $ and $A_1 = A_2=-
1 /2 $, and so
\begin{gather*}
z \big(Y^{\mathsf L}_n \big)^{\prime \prime} + \big(Y^{\mathsf L}_n \big)^{\prime} \begin{bmatrix} 
 \alpha + 1 -  z & 0 \\
0 & 1-\alpha +  z
 \end{bmatrix} 
+ Y^{\mathsf L}_n 
 \begin{bmatrix} 
 -
 1 / 2 
 & 0 \\
0 & 
1 / 2 
 \end{bmatrix} 
= \begin{bmatrix} 
- 
({ n +1})/{2}  
  & 
 0 
 \\
 0 & 
- 
({ n -1})/{2}  
\end{bmatrix} Y^{\mathsf L}_n (z) ,
\end{gather*} 
and so 
we get the second order equations for the $ \big\{ P_n \big\}_{ n  \in \mathbb N} $ (cf. for example~\cite{chihara}) and $\big\{ Q_n \big\}_{ n  \in \mathbb N} $ in the Laguerre case, i.e. for all $n \in \mathbb N$ we have
\begin{align*}
z P_n^{\prime\prime} (z) - ( z - \alpha -1) P_n^{\prime} (z) &= - n  P_n (z) , \\ 
z Q_n^{\prime\prime} (z) + ( z - \alpha +1 ) Q_n^{\prime} (z) &= - ( n + 1 ) Q_n (z) .
\end{align*}

\subsection{Adjoint operators}

\begin{defi}
Given linear operator $L : \C^{N\times N}[z] \to \C^{N\times N}[z]
 $ 
and a matrix of weights $ W$, its adjoint operator $L^*$ is an operator such that
\begin{align*}
\langle L(P), \tilde P \rangle_{ W}&=\langle P, L^*(\tilde P )\rangle_{ W}, & 
P(z),\tilde P(z)\in \C^{N\times N}[z],
\end{align*}
in terms of the sesquiliner form introduced in~\eqref{eq:sesquilinear}.
\end{defi}

Care must be taken at this point because in this definition of adjoint of a matrix differential operator we are not taken the transpose or the Hermitian conjugate of the matrix coefficients as was done in~\cite{duran_3}.

\begin{defi}
Motivated by~\eqref{eq:secondorderpl} and~\eqref{eq:secondorderpr} we introduce two linear operators $\pmb \ell^{\mathsf L}$ and $\pmb \ell^{\mathsf R}$, acting on the linear space of polynomials $\C^{N \times N} [z]$ as follows
\begin{align*}
\pmb \ell^{\mathsf L} (P) :
=z P^{\prime\prime} + P^{\prime} a^{\mathsf L}(z) + P b^{\mathsf L}(z) ,
 &&
\pmb \ell^{\mathsf R} (P) :
= z P^{\prime\prime} + a^{\mathsf R}(z) P^{\prime} + b^{\mathsf R}(z) P .
\end{align*}
where
$a^{\mathsf L}(z) 
 : = 2 h^{\mathsf L} + I_N $, 
$ b^{\mathsf L}(z) 
 : = \mathcal N ( h^{\mathsf L}(z) )$, 
$ a^{\mathsf R}(z) 
 : = 2 h^{\mathsf R} + I_N$, 
$ b^{\mathsf R}(z) 
: = \mathcal N ( h^{\mathsf R}(z) )$.
\end{defi}

\begin{pro} \label{prop:eqrho}
Let us assume that the matrix of weights $ W 
$ do satisfy the fol\-low\-ing boundary conditions
\begin{align}\label{eq:bconditions}
 zW \big|_0^\infty &=0_N, & \big( \big(z W \big)^{\prime} - a^{\mathsf L} W \big) \big|_0^\infty &=0_N , & \big( \big(z W \big)^{\prime} -W a^{\mathsf R} \big) \big|_0^\infty &=0_N.
 \end{align}
 Here 
$
\displaystyle
f(z)\big|_0^\infty:=\lim_{z\to\infty}f(z)- \lim_{z\to 0}f(z)
$.
Then,
$ W$ satisfies a Pearson differential equation~\eqref{eq:Pearson} if, and only if,
$ W$ satisfies the fol\-low\-ing second order matrix differential equations
\begin{align}
\label{eq:rhoml}
\big( z W \big)^{\prime \prime} - \big( a^{\mathsf L} W\big)^{ \prime} + b^{\mathsf L} W = W b^{\mathsf R} , \\
\label{eq:rhomr}
\big( z W \big)^{\prime \prime} - \big( W a^{\mathsf R} \big)^{ \prime} + W b^{\mathsf R} =b^{\mathsf L} W . 
\end{align}
\end{pro}

\begin{proof}
Taking derivative on~\eqref{eq:Pearson}, we get 
\begin{align*}
{ W}^{\prime \prime} & = \Big( \frac{h^{\mathsf L}}{z} \Big)^{\prime} W + \frac{h^{\mathsf L}}{z} W^{\prime}+ W^{\prime} \frac{h^{\mathsf R}}{z} +
 W \Big( \frac{h^{\mathsf R}}{z} \Big)^{\prime} 
 \\
 & = \Big( \frac{ (h^{\mathsf L})^{\prime}}{z} - \frac{h^{\mathsf L}}{z^2} \Big) W +\frac{h^{\mathsf L}}{z} \Big( \frac{h^{\mathsf L}}{z} W + W \frac{h^{\mathsf R}}{z} \Big)+ \Big( \frac{h^{\mathsf L}}{z} W + W \frac{h^{\mathsf R}}{z} \Big) \frac{h^{\mathsf R}}{z}+
 W \Big(\frac{ (h^{\mathsf R} )^{\prime}}{z} -\frac{h^{\mathsf R}}{z^2} \Big),
\end{align*}
so it holds that
\begin{align*}
(z W)^{\prime \prime} & = 2 { W}^{\prime } + z { W}^{\prime \prime} 
 \\
 & =
2 { W}^{\prime } + \Big( (h^{\mathsf L})^{\prime} - \frac{h^{\mathsf L}}{z} \Big) W + h^{\mathsf L} \Big( \frac{h^{\mathsf L}}{z} W + W \frac{h^{\mathsf R}}{z} \Big)+ \Big( \frac{h^{\mathsf L}}{z} W + W \frac{h^{\mathsf R}}{z} \Big) h^{\mathsf R}+
 W \Big( (h^{\mathsf R} )^{\prime} -\frac{h^{\mathsf R}}{z} \Big)\\
 & = 
{ W}^{\prime } + b^{\mathsf L} W + W b^{\mathsf R}+ \frac{2}{z} h^{\mathsf L} W h^{\mathsf R}.
\end{align*}
But, it is easy to see that
\begin{align*}
\big( a^{\mathsf L} W\big)^{ \prime} = 2 (h^{\mathsf L})^{\prime}W + 2 \frac{(h^{\mathsf L})^2}{z} W + \frac{2}{z} h^{\mathsf L}W h^{\mathsf R} + W^{\prime}
= { W}^{\prime } + 2 b^{\mathsf L} W + \frac{2}{z} h^{\mathsf L} W h^{\mathsf R} ,
\end{align*}
and
\begin{align*}
\big( W a^{\mathsf R} \big)^{ \prime} = 2W (h^{\mathsf R})^{\prime} + 2 W \frac{(h^{\mathsf R})^2}{z} + \frac{2}{z} h^{\mathsf L}W h^{\mathsf R} + W^{\prime}
\big( W a^{\mathsf R} \big)^{ \prime} = { W}^{\prime } + 2 W b^{\mathsf R}+ \frac{2}{z} h^{\mathsf L} W h^{\mathsf R},
\end{align*}
and so we arrive to~\eqref{eq:rhoml} and~\eqref{eq:rhomr}.

The reciprocal result is a consequence of adding the equations~\eqref{eq:rhoml},~\eqref{eq:rhomr} and using the boundary conditions~\eqref{eq:bconditions}.
\end{proof}

Now, we will see that these two operators are adjoint to each other with respect to the sesquilinear form induced by the weight functions $ W$.
\begin{pro} \label{prop:aa}
Whenever $ W $ satisfies~\eqref{eq:Pearson} and the boundary conditions~\eqref{eq:bconditions}, we have that
\begin{align}\label{eq:adjoint_ell}
 \pmb \ell^{\mathsf R}= \big(\pmb \ell^{\mathsf L}\big)^*.
\end{align} 
\end{pro}

\begin{proof}
By using the linearity of these operators it is sufficient to prove
\begin{align*}
\langle \pmb \ell^{\mathsf L} ( P_n^{\mathsf L} ) \, , \, P_k^{\mathsf R} \rangle_ W = \langle P_n^{\mathsf L} \, , \, \pmb \ell^{\mathsf R} ( P_k^{\mathsf R} ) \rangle_ W , && n , k \in \mathbb N . 
\end{align*}
If we omit, for the sake of simplicity, the $z$ dependence of the integrands in the integrals, we have
\begin{align*}
\langle \pmb \ell^{\mathsf L} ( P_n^{\mathsf L} ) \, , \, P_k^{\mathsf R} \rangle_ W 
 = 
\int_{\gamma} z (P_n^{\mathsf L})^{\prime\prime} \, W \, P_k^{\mathsf R} \, \operatorname d z
+ \int_{\gamma} (P_n^{\mathsf L})^{\prime} \, a^{\mathsf L} \, W \, P_k^{\mathsf R} \, \operatorname d z
+ \int_{\gamma} P_n^{\mathsf L} \, 
 b^{\mathsf L} \, W \, P_k^{\mathsf R} \, \operatorname d z ,
\end{align*}
and, using integration by parts, we find
\begin{align*}
\langle \pmb \ell^{\mathsf L} ( P_n^{\mathsf L} ) , P_k^{\mathsf R} \rangle_ W
 & = \big( z (P_n^{\mathsf L})^{\prime} W P_k^{\mathsf R}\big) \big|_{\partial\gamma} 
 - \int_{\gamma} (P_n^{\mathsf L})^{\prime} \Big( \big( z W P_k^{\mathsf R}\big)^{\prime} - a^{\mathsf L} W P_k^{\mathsf R} \Big) \d z
+ \int_{\gamma} P_n^{\mathsf L} 
b^{\mathsf L} W P_k^{\mathsf R} \d z
 \\
 & \hspace{-.085cm} 
\begin{multlined}[t][0.8\textwidth]
 =
 \big( z(P_n^{\mathsf L})^{\prime} W P_k^{\mathsf R}\big) \big|_{\partial\gamma} - \Big(P_n^{\mathsf L} \Big( \big(z W P_k^{\mathsf R}\big)^{\prime} - a^{\mathsf L} WP_k^{\mathsf R} \Big) \Big) \Big|_{\partial\gamma} \\
+ \int_{\gamma} P_n^{\mathsf L} \, \big( (z W \, P_k^{\mathsf R} )^{\prime\prime}
- ( a^{\mathsf L} \, W \, P_k^{\mathsf R} )^{\prime}
+b^{\mathsf L} \, W \, P_k^{\mathsf R} \big) 
\, \operatorname d z .
 \end{multlined}
\end{align*}
Now, considering the boundary conditions~\eqref{eq:bconditions}
and taking into account that
\begin{align*}
\big( z W \, P_k^{\mathsf R} \big)^{\prime\prime} & = \big( z W \big)^{\prime\prime} \, P_k^{\mathsf R} 
+ 2 \, \big( z W \big) ^{\prime} \, (P_k^{\mathsf R})^{\prime} + z W \, (P_k^{\mathsf R})^{\prime\prime} , 
 \\
( a^{\mathsf L} \, W \, P_k^{\mathsf R} )^{\prime} & = (a^{\mathsf L} \, W)^{\prime} \, P_k^{\mathsf R}
+ ( a^{\mathsf L} \, W) \, (P_k^{\mathsf R})^{\prime} ,
 \end{align*} 
we arrive to
\begin{multline*}
\langle \pmb \ell^{\mathsf L} ( P_n^{\mathsf L} ) \, , \, P_k^{\mathsf R} \rangle_ W 
 =
\int_{\gamma} P_n^{\mathsf L} 
\big( \big( z W \big)^{\prime\prime} 
- (a^{\mathsf L} \, W)^{\prime}
+ b^{\mathsf L} W 
\big) P_k^{\mathsf R} \d z
 \\ +
 \int_{\gamma} P_n^{\mathsf L} 
\big(2 \big( z W \big)^{\prime} 
- a^{\mathsf L} W
\big) ( P_k^{\mathsf R} )^{\prime} \d z
+ \int_{\gamma} P_n^{\mathsf L}z W 
( P_k^{\mathsf R} )^{\prime\prime} \d z ,
\end{multline*}
and so
\begin{align*}
\langle \pmb \ell^{\mathsf L} ( P_n^{\mathsf L} ) \, , \, P_k^{\mathsf R} \rangle_ W 
 =
\int_{\gamma} P_n^{\mathsf L} W 
\big( z (P_k^{\mathsf R} )^{\prime\prime}
+ \, a^{\mathsf R} (P_k^{\mathsf R})^{\prime}
+ b^{\mathsf R} \big) P_k^{\mathsf R}
 \d z, && 
n,k \in \in \N 
,
\end{align*}
or, equivalently,
\begin{align*}
\langle \pmb \ell^{\mathsf L} ( P_n^{\mathsf L} ) \, , \, P_k^{\mathsf R} \rangle_ W 
 =
\langle P_n^{\mathsf L} \, , \, \pmb \ell^{\mathsf R} ( P_k^{\mathsf R} ) \rangle_ W 
 , 
\end{align*}
which completes the proof.
\end{proof}

\begin{defi}
 Let $\alpha^\mathsf L$ and $\alpha^\mathsf R$ be two $N \times N$ matrices and define the fol\-low\-ing linear operators acting on the space of matrix polynomials $ \mathbb \C^{N\times N}[z]$ as follows
 \begin{align*} 
 {\mathcal L}^{\mathsf L} (P) &:= z P^{\prime\prime} + P^{\prime} a^{\mathsf L} + P {\alpha}^{\mathsf L}, &
 {\mathcal L}^{\mathsf R} (P) &:=z P^{\prime\prime} + a^{\mathsf R} P^{\prime} + {\alpha}^{\mathsf R} P.
 \end{align*}
\end{defi}
Observe that
\begin{align*} 
{\mathcal L}^{\mathsf L} (P) &= \pmb \ell^{\mathsf L} (P)- P \,b^{\mathsf L} + P {\alpha}^{\mathsf L},& 
{\mathcal L}^{\mathsf R} (P) &=\pmb \ell^{\mathsf R} (P )- b^{\mathsf R} P + {\alpha}^{\mathsf R} P.
\end{align*}
We have the fol\-low\-ing characterization.

\begin{teo} \label{teo:novo}
The fol\-low\-ing conditions are equivalent:

\hangindent=.9cm \hangafter=1
{\noindent}$\phantom{ol}${\rm i)}
${\mathcal L}^{\mathsf R}=\big({\mathcal L}^{\mathsf L} \big)^*$ with respect to the matrix of weights 
$ W $.

\hangindent=.9cm \hangafter=1
{\noindent}$\phantom{ol}${\rm ii)}
The matrix of weights $ W$ satisfies the matrix Pearson equation~\eqref{eq:Pearson} with the boundary conditions~\eqref{eq:bconditions} as well as fulfills the constraint
\begin{align}
\label{eq:escondida}
\big( {\alpha}^{\mathsf L} - \,b^{\mathsf L} \big) W = W \big( {\alpha}^{\mathsf R} - b^{\mathsf R} \big). 
\end{align}

\hangindent=.9cm \hangafter=1
{\noindent}$\phantom{ol}${\rm iii)}
The matrix of weights $ W $ satisfies the boundary conditions~\eqref{eq:bconditions} as well as 
\begin{align}
\label{eq:rhoalphal}
( z W)^{\prime \prime} -\big(a^{\mathsf L} W\big)^{ \prime} 
+ \alpha^{\mathsf L} W = W \alpha^{\mathsf R} , \\
\label{eq:rhoalphar}
( z W)^{\prime \prime} - \big( W a^{\mathsf R} \big)^{ \prime} 
+ W \alpha^{\mathsf R} = \alpha^{\mathsf L} W . 
\end{align}
\end{teo}

\begin{proof}
{}From Proposition~\ref{prop:aa}
\begin{align*}
\langle {\mathcal L}^{\mathsf L} ( P ) , \tilde P \rangle_ W = \langle P , {\mathcal L}^{\mathsf R} (\tilde P ) \rangle_ W ,
\end{align*}
if and only if
\begin{align*}
\langle- P \,b^{\mathsf L} + P {\alpha}^{\mathsf L} , \tilde P \rangle_ W = \langle P , - b^{\mathsf R} \tilde P + {\alpha}^{\mathsf R} \tilde P \rangle_ W ,
\end{align*}
that is~\eqref{eq:escondida} takes place, and so i) is equivalent to ii).

To prove that i) is equivalent to iii) observe that, adding~\eqref{eq:rhoalphal} and~\eqref{eq:rhoalphar}, the fol\-low\-ing~holds
\begin{align*}
z { W}^{\prime \prime} = \big( a^{\mathsf L} W \big)^{ \prime} +
\big( W a^{\mathsf R} \big)^{ \prime} , 
\end{align*} 
which transforms~\eqref{eq:Pearson} if we integrate requesting boundary conditions~\eqref{eq:bconditions}.
Moreover, if we subtract~\eqref{eq:rhoalphal} and~\eqref{eq:rhoalphar} we arrive directly to~\eqref{eq:escondida}.
\end{proof}

\subsection{Eigenvalue problems}




Now we discuss how our findings based on the Riemann--Hilbert problem are linked with previous  results by   Durán and Grünbaum~\cite{duran_3,duran2004,duran_1,duran_2}.
The next theorem shows when the polynomials and associated functions of second kind are eigenfunctions of a second order operator.

\begin{teo}[Eigenvalue problems for Laguerre matrix orthogonal polynomials] \label{teo:nuevo}
Let $h^\mathsf L$ and $h^\mathsf R$ be degree one matrix polynomials,~{i.e.}
\begin{align*}
h^{\mathsf L}(z) = A^{\mathsf L} z + B^{\mathsf L} , &&
h^{\mathsf R}(z) = A^{\mathsf R} z +B^{\mathsf R} , &&
A^{\mathsf L}, A^{\mathsf R}, B^{\mathsf L}, B^{\mathsf R} \in \mathbb C^{N \times N} ,
\end{align*} 
with $A^{\mathsf L}, A^{\mathsf R}$ definite negative, and $ W$ a matrix of weights a solution of~\eqref{eq:rhoalphal},~\eqref{eq:rhoalphar} subject to the boundary conditions~\eqref{eq:bconditions}.
Then, the fol\-low\-ing conditions are equivalent:

\hangindent=.9cm \hangafter=1
{\noindent}$\phantom{ol}${\rm i)} \label{1} 
The operators ${\mathcal L}^{\mathsf L}$ and ${\mathcal L}^{\mathsf R} $ are adjoint operators with respect to the matrix of weights $ W$,
i.e. ${\mathcal L}^{\mathsf R}=\big({\mathcal L}^{\mathsf L}\big)^*$.

\hangindent=.9cm \hangafter=1
{\noindent}$\phantom{ol}${\rm ii)} \label{2} 
The biorthogonal polynomial sequences with respect to $ W$, say
$\big\{ P_n^{\mathsf L} \big\}_{n\in\N}$, $ \big\{ P_n^{\mathsf R}\big\}_{n\in\N}$, 
are eigenfunctions of ${\mathcal L}^{\mathsf L}$ and ${\mathcal L}^{\mathsf R}$,~i.e. there exist $N \times N $ matrices, $\lambda_n^{\mathsf L}$, $\lambda_n^{\mathsf R}$ such that
\begin{align*} 
{\mathcal L}^{\mathsf L} (P_n^{\mathsf L})
 &= \lambda_n^{\mathsf L} P_n^{\mathsf L}, &
 {\mathcal L}^{\mathsf R} (P_n^{\mathsf R}) &= P_n^{\mathsf R} \lambda_n^{\mathsf R},
\end{align*}
with $\lambda_n^{\mathsf L} C_n^{-1} = C_n^{-1} \lambda_n^{\mathsf R}$, $n \in \mathbb N$.

\hangindent=.9cm \hangafter=1
{\noindent}$\phantom{ol}${\rm iii)} \label{3} 
The functions of second kind, $ \big\{ Q_n^\mathsf L \big\}_{n \in \mathbb N}$ and $ \big\{ Q_n^\mathsf R\big\}_{n \in \mathbb N}$, associated with the biorthogonal polynomials, $\big\{ P_n^\mathsf L \big\}_{n \in \mathbb N}$ 
and $\big\{ P_n^\mathsf R \big\}_{n \in \mathbb N}$,
fulfill the second order differential~equations,
\begin{align*} 
z\big(Q_n^{\mathsf L} \big)^{\prime \prime}(z) 
+\big(Q_n^{\mathsf L} \big)^{ \prime}(z) \, \big( -2 h^{\mathsf R} (z) + I_N \big)
+ Q_n^{\mathsf L} (z) \, (\alpha^{\mathsf R} - 2 A^{\mathsf R} ) 
 &= \lambda_n^{\mathsf L} \, Q_n^{\mathsf L} (z),
 \\
z\big(Q_n^{\mathsf R} \big)^{\prime \prime}(z) 
+\big( - 2 h^{\mathsf L} (z) + I_N \big) \big(Q_n^{\mathsf R} \big)^{ \prime}(z) 
 +
(\alpha^{\mathsf L} - 2 A^{\mathsf L} ) \, Q_n^{\mathsf R} (z) 
 &= 
Q_n^{\mathsf R} (z) \, \lambda_n^{\mathsf R}.
\end{align*}
\end{teo}

\begin{proof}
The proof follows from similar arguments as in~\cite[Theorem~5]{BFM}.
\end{proof}


The interpretation in terms of adjoint operators, inherits from the Riemann--Hilbert problem the characterization for the $\big\{Q_n^\mathsf L \big\}_{n \in \N}$ and 
$\big\{Q_n^\mathsf R \big\}_{n \in \N}$ that resembles the ones in~\eqref{eq:secondorderql} and~\eqref{eq:secondorderqr}. Moreover, Theorems~\ref{teo:novo} we see that $W$ in Theorem~\ref{teo:nuevo} can be taken as a solution of a Pearson Sylvester differential equation like~\eqref{eq:Pearson} and satisfying~\eqref{eq:escondida}.

\subsection{
Reductions}

We consider two possible reductions for the matrix of weights, the symmetric reduction and the Hermitian reduction.
\begin{enumerate}
\item A matrix of weights $W(z)$ with support on $\gamma$ is said to be symmetric if
\begin{align*}
(W(z))^\top&=W(z),& z&\in\gamma.
\end{align*}
\item A matrix of weights $W(x)$ with support on $\R$ is said to be Hermitian if
\begin{align*}
(W(x))^\dagger&=W(x),& x&\in\R.
\end{align*}
\end{enumerate}
These two reductions lead to orthogonal polynomials, as the two biorthogonal families are identified,~i.e., for the symmetric case
$P_n^\mathsf R (z)=\big(P_n^\mathsf L (z)\big)^\top$,
$Q_n^\mathsf R (z)=\big(Q_n^\mathsf L (z)\big)^\top$,
and for the Hermitian case, with $\gamma=\R$,
$P_n^\mathsf R (z)=\big(P_n^\mathsf L (\bar z)\big)^\dagger$, 
$Q_n^\mathsf R (z)=\big(Q_n^\mathsf L (\bar z)\big)^\dagger$.
In both cases biorthogonality collapses into orthogonality.

For the symmetric or Hermitian reductions we find that
\begin{align*}
\pmb \ell^{\mathsf R} (P) 
& = \big(\pmb \ell^{\mathsf L}(P^\top)\big)^\top, && \text{symmetric}, 
\\
\pmb \ell^{\mathsf R} (P) 
& = \big(\pmb \ell^{\mathsf L}(P^\dagger)\big)^\dagger, && \text{Hermitian}, 
\end{align*}
where in the last case we take $x\in\R$. Relation~\eqref{eq:adjoint_ell} reads in this case as follows
\begin{align*}
\pmb\ell^* (P) & =(\pmb\ell(P^\top))^\top, &&
\text {symmetric}, \\
\pmb\ell^* (P) & =(\pmb\ell(P^\dagger))^\dagger, && 
\text {Hermitian} ,
\end{align*}
for $P$ any matrix polynomial and $\pmb\ell:=\pmb \ell^{\mathsf L}$.

We find that
\begin{align*}
\mathcal L^{\mathsf R} (P) & = \big( \mathcal L^{\mathsf L}(P^\top)\big)^\top, && \text{symmetric}, \\
\mathcal L^{\mathsf R} (P) & = \big( \mathcal L^{\mathsf L}(P^\dagger)\big)^\dagger, && \text{Hermitian},
\end{align*}
where in the last case we take $x\in\R$.
\\ 
Moreover, the fol\-low\-ing are equivalent conditions

\begin{enumerate}
\item 
Equations 
\begin{align}\label{eq:adjoint_reduced}
\mathcal L^* (P) 
=( \mathcal L(P^\top))^\top, && 
\text {symmetric}, \\
\mathcal L^* (P) =( \mathcal L(P^\dagger))^\dagger, && 
\text {Hermitian} ,
\end{align}
are satisfied by any matrix polynomial $P$, where $ \mathcal L:= \mathcal L^{\mathsf L}$.

\item The matrix of weights $ W $ satisfies the matrix Pearson equation
\begin{align}
z W^{\prime}(z) & = h(z) W(z)+ W(z) (h(z))^\top, && \text{symmetric}, \\
z W^{\prime}(z) & = h(z) W(z)+ W(z) (h(\bar z))^\dagger, && \text{Hermitian}.
\end{align}
with the boundary conditions
\begin{align}
\label{eq:boundaryc_reduced}
 zW \big|_0^\infty = 0_N, 
&& \big( \big(z W \big)^{\prime} - (2 h+ I_N ) W \big) \big|_0^\infty = 0_N 
\end{align}
as well as fulfills the constraint
\begin{align*}
\big( {\alpha}- \, \mathcal N (h(z) ) W (z) & = W (z) \big( {\alpha}^\top - \mathcal N (h(z) )^\top \big), && \text{symmetric}, \\
\big( {\alpha}- \, \mathcal N (h(z) ) \big) W (z) & = W (z) \big( {\alpha}^\dagger - \mathcal N (h(z) )^\dagger \big), && \text{Hermitian},
\end{align*}

\item 
The matrix of weights $ W $ satisfies the boundary conditions~\eqref{eq:boundaryc_reduced}
as well as 
\begin{align}\label{eq:second_order_reduced}
(z W (z) )^{\prime \prime} - \big( (2 h (z) + I_N ) W (z) \big)^{ \prime} 
+ \alpha W (z) & = W (z) \alpha^\top , && \text{symmetric},\\
(z W (z) )^{\prime \prime} - \big( (2 h (z) + I_N )W (z) \big)^{ \prime} 
+ \alpha W (z) & = W (z) \alpha^\dagger, && \text{Hermitian}.
\end{align}
\end{enumerate}

For the symmetric or Hermitian reductions we take
$ h(z) = A z + B $,
 with $A$ definite negative, and $ W$ a matrix of weights a solution of~\eqref{eq:second_order_reduced} subject to the boundary conditions~\eqref{eq:boundaryc_reduced}.
 Then, the fol\-low\-ing conditions are equivalent:
 
 \hangindent=.9cm \hangafter=1
 {\noindent}$\phantom{ol}${\rm i)} \label{1} 
Equation~\eqref{eq:adjoint_reduced} is satisfied.
 
 \hangindent=.9cm \hangafter=1
 {\noindent}$\phantom{ol}${\rm ii)} \label{2} 
 The matrix orthogonal polynomials with respect to $ W$
 are eigenfunctions of ${\mathcal L}$.
 
\hangindent=.9cm \hangafter=1
{\noindent}$\phantom{ol}${\rm iii)} \label{3} 
The functions of second kind, $ \big\{ Q_n(z) \big\}_{n \in \mathbb N}$, associated with the matrix orthogonal polynomials, $\big\{ P_n (z) \big\}_{n \in \mathbb N}$ fulfill the second order differential~equations,
 \begin{align*}
 z\big(Q_n\big)^{\prime \prime}(z) 
 + \big(Q_n\big)^{ \prime}(z) \, ( -2 h (z) + I_N )^\top
 + Q_n (z) \, (\alpha^\top - 2 A^\top ) 
 = \lambda_n \, Q_n (z), && \text{symmetric}, 
 \\
 z \big(Q_n\big)^{\prime \prime}(z) 
 + \big(Q_n\big)^{ \prime}(z) \, ( -2 h (z) + I_N )^\dagger
 + Q_n (z) \, (\alpha^\dagger - 2 A^\dagger ) 
 = \lambda_n \, Q_n (z), && \text{Hermitian}.
 \end{align*}

These equivalences, excluding the one for the second kind functions (which is new), coincide with those of~\cite{duran2004}.
Therefore, these results could be understood as an extension of those obtained by Durán and Grünbaum to the non Hermitian orthogonality scenario.

\section{Matrix discrete Painlev\'e~IV }

We can consider, using the notation introduced before, the matrix weight measure $W = W_{\mathsf L} W_{\mathsf R} $
such that
\begin{align*}
z (W^\mathsf L)^{\prime}(z) =( A_{\mathsf L} +
B_{\mathsf L} z + C_{\mathsf L} z^2) W^{\mathsf L}(z), &&
z(W^\mathsf R)^{\prime}(z)
= W^{\mathsf R}(z) (A_{\mathsf R} +
B_{\mathsf R} z + C_{\mathsf R} )z^2 . 
\end{align*}
From Theorem~\ref{prop:mn} we get that the matrix
\begin{align*} 
\tilde{M}_n=z M_n^{\mathsf L} 
\end{align*}
is given explicitly by
\begin{align*}
(\tilde{M}_n^{\mathsf L})_{11} & = C_{n}^{-1} C_{\mathsf R} C_{n-1} +
(A_{\mathsf L} + B_{\mathsf L} z + C_{\mathsf L} z^2) + B_{\mathsf L}
q_{\mathsf R,n-1}^{1} + p_{\mathsf L,n}^{1} B_{\mathsf L} \\
& \phantom{olaolaolaola} + z
(C_{\mathsf L}q_{\mathsf R,n-1}^{1} + p_{\mathsf L,n}^{1} C_{\mathsf L})
+ C_{\mathsf L} q_{\mathsf R,n-1}^{2}
+ p_{\mathsf L,n}^{2} C_{\mathsf L}
+ p_{\mathsf L,n}^{1} C_{\mathsf L} q_{\mathsf R,n-1}^{1} , \\
(\tilde{M}_n^{\mathsf L})_{12} & = (B_{\mathsf L} + C_{\mathsf L} z +
C_{\mathsf L} q_{\mathsf R,n}^{1} + p_{\mathsf L,n}^{1} C_{\mathsf L})
C_n^{-1} + C_n^{-1} (B_{\mathsf R} + C_{\mathsf R} z + C_{\mathsf R}
p_{\mathsf R,n}^{1} + q_{\mathsf L,n}^{1} C_{\mathsf R}) ,
\\
(\tilde{M}_n^{\mathsf L})_{21} & = -C_{n-1} (B_{\mathsf L} + C_{\mathsf L} z +
C_{\mathsf L} q_{\mathsf R,n-1}^{1} + p_{\mathsf L,n-1}^{1} C_{\mathsf
L}) - (B_{\mathsf R} + C_{\mathsf R} z + C_{\mathsf R} p_{\mathsf
R,n-1}^{1} + q_{\mathsf L,n-1}^{1} C_{\mathsf R}) C_{n-1} ,
\\
(\tilde{M}_n^{\mathsf L})_{22} & = -C_{n-1} C_{\mathsf L} C_n^{-1} -
(A_{\mathsf R} + B_{\mathsf R} z + C_{\mathsf R} z^2) - B_{\mathsf R}
p_{\mathsf R,n}^{1} - q_{\mathsf L,n-1}^{1} B_{\mathsf R} \\
& \phantom{olaolaolaola} - z
(C_{\mathsf R}p_{\mathsf R,n}^{1}
+ q_{\mathsf L,n-1}^{1} C_{\mathsf R})
- C_{\mathsf R} p_{\mathsf R,n}^{2}
- q_{\mathsf L,n-1}^{2} C_{\mathsf R}
- q_{\mathsf L,n-1}^{1} C_{\mathsf R} p_{\mathsf R,n}^{1} .
\end{align*}
From the three term recurrence relation for $\{ P_n^{\mathsf L} \}_{n \in \mathbb N}$ we get that 
$p_{\mathsf L,n}^1 - p_{\mathsf L,n+1}^1 = \beta_n^{\mathsf L}$ and $p_{\mathsf L,n}^2 - p_{\mathsf L,n+1}^2 = \beta_n^{\mathsf L} p_{\mathsf L,n}^1 + \gamma_n^{\mathsf L}$
where $\gamma_n^{\mathsf L} = C_{n}^{-1}C_{n-1} $.
Consequently,
\begin{align*}
p_{\mathsf L,n}^1 &= - \sum_{k=0}^{n-1} \beta_k^{\mathsf L} , &
p_{\mathsf L,n}^2 &=
\sum_{i,j=0}^{n-1} \beta_i^{\mathsf L} \beta_j^{\mathsf L} - \sum_{k=0}^{n-1}\gamma_k^{\mathsf L}.
\end{align*}

In the same manner, from the three term recurrence relation for $\{ Q_n^{\mathsf L} \}_{n \in \mathbb N}$ we deduce that
$q_{\mathsf L,n}^1 - q_{\mathsf L,n-1}^1 = \beta_n^{\mathsf R}:=C_n \beta_n^{\mathsf L} C_n^{-1}$ and $q_{\mathsf L,n}^2 - q_{\mathsf L,n-1}^2 = \beta_n^{\mathsf R} q_{\mathsf L,n}^1 + \gamma_n^{\mathsf R}$,
where $\gamma_n^{\mathsf R} = C_n C_{n+1}^{-1}$.

If we consider that $W = W^{\mathsf L}$ and $ W^{\mathsf R} = I_N$, and use the representation for $ \{ P_n^{\mathsf L} \}_{n \in \N}$ and $\{ Q_n^{\mathsf L} \}_{n \in \N}$ in $z$ powers, the $(1,2)$ and $(2,2)$ entries in ~\eqref{eq:firstodeYnL} read
\begin{align*}
& (2n+1) I_N + A_{\mathsf L} 
+ C_{\mathsf L} (\gamma_{n+1}^{\mathsf L} + \gamma_{n}^{\mathsf L} + (\beta_{n}^{\mathsf L})^2) + B_{\mathsf L} \beta_{n}^{\mathsf L} 
= [p_{\mathsf L,n}^1, C_{\mathsf L}] p_{\mathsf L,n+1}^1
- [p_{\mathsf L,n}^2, C_{\mathsf L}] 
- [p_{\mathsf L,n}^1, B_{\mathsf L}], & \\
& \beta_n^{\mathsf L} = \gamma_{n}^{\mathsf L} \big( C_{\mathsf L}(\beta_n^{\mathsf L} + \beta_{n-1}^{\mathsf L}) + [p_{\mathsf L,n-1}^1, C_{\mathsf L}] + B_{\mathsf L} \big) -
\big( C_{\mathsf L}(\beta_n^{\mathsf L} + \beta_{n+1}^{\mathsf L}) + [p_{\mathsf L,n}^1, C_{\mathsf L}] + B_{\mathsf L} \big) \gamma_{n+1}^{\mathsf L} . 
\end{align*}
We can write these equations as follows
\begin{align}\label{eq:dPIV_with_B_1}
\begin{multlined}[t][.9\textwidth]
(2n+1) I_N + A_{\mathsf L} 
+ C_{\mathsf L} (\gamma_{n+1}^{\mathsf L} + \gamma_{n}^{\mathsf L} ))
+(C_{\mathsf L} \beta_{n}^{\mathsf L}+ B_{\mathsf L}) \beta_{n}^{\mathsf L} 
 \\
= 
 \Big[ \sum_{k=0}^{n-1} \beta_k^{\mathsf L}, C_{\mathsf L}\Big] \sum_{k=0}^{n} \beta_k^{\mathsf L}
- \Big[\sum_{i,j=0}^{n-1} \beta_i^{\mathsf L} \beta_j^{\mathsf L} - \sum_{k=0}^{n-1}\gamma_k^{\mathsf L}
, C_{\mathsf L}\Big] - \Big[ \sum_{k=0}^{n-1} \beta_k^{\mathsf L}, B_{\mathsf L}\Big],\end{multlined} \\
\label{eq:dPIV_with_B_2}
\begin{multlined}[t][.9\textwidth]
\beta_n^{\mathsf L}-\gamma_{n}^{\mathsf L} \big( C_{\mathsf L}(\beta_n^{\mathsf L} + \beta_{n-1}^{\mathsf L}) + B_{\mathsf L} \big) + \big( C_{\mathsf L}(\beta_n^{\mathsf L} + \beta_{n+1}^{\mathsf L}) + B_{\mathsf L} \big) \gamma_{n+1}^{\mathsf L} \\= -\gamma_{n}^{\mathsf L} \Big[ \sum_{k=0}^{n-1} \beta_k^{\mathsf L} , C_{\mathsf L}\Big] +
\Big[-\sum_{k=0}^{n-1} \beta_k^{\mathsf L} , C_{\mathsf L}\Big] \gamma_{n+1}^{\mathsf L}.
\end{multlined}
\end{align}

We will show now that this system contains a noncommutative version of an instance of discrete Painlev\'e~IV equation, as happens in the analogous case for the scalar scenario.

We see, on the \emph{r.h.s.} of the nonlinear discrete equations \eqref{eq:dPIV_with_B_1} and \eqref{eq:dPIV_with_B_2} nonlocal terms (sums) in the recursion coefficients $\beta_n^{\mathsf L}$ and $\gamma_n^{\mathsf L}$, all of them inside commutators. 
These nonlocal terms vanish whenever the three matrices $\{ A_{\mathsf L}, B_{\mathsf L}, C_{\mathsf L}\}$ conform an Abelian set. Moreover, $\{ A_{\mathsf L}, B_{\mathsf L}, C_{\mathsf L},\beta_{n}^{\mathsf L}, \gamma_{n}^{\mathsf L}\}$ is also an Abelian set. In this commutative setting we have
\begin{gather*}
(2n+1) I_N + A_{\mathsf L} 
+ C_{\mathsf L} (\gamma_{n+1}^{\mathsf L} + \gamma_{n}^{\mathsf L} ))
+(C_{\mathsf L} \beta_{n}^{\mathsf L}+ B_{\mathsf L}) \beta_{n}^{\mathsf L}= 0_N,\\
\beta_n^{\mathsf L}-\gamma_{n}^{\mathsf L} \big( C_{\mathsf L}(\beta_n^{\mathsf L} + \beta_{n-1}^{\mathsf L}) + B_{\mathsf L} \big) + \big( C_{\mathsf L}(\beta_n^{\mathsf L} + \beta_{n+1}^{\mathsf L}) + B_{\mathsf L} \big) \gamma_{n+1}^{\mathsf L} =0_N.
\end{gather*}

In terms of
$
\displaystyle 
\xi_n := \frac{ A_{\mathsf L} }{2}+ n I_N+C_{\mathsf L} \gamma_n$
and $ \mu_n := C_{\mathsf L} \beta_n ^{\mathsf L} +B_{\mathsf L} $
the above equations are
\begin{align*}
\beta^{\mathsf L}_n \mu_n& = - (\xi_n + \xi_{n+1}) , 
 &
\beta^{\mathsf L}_n (\xi_n -\xi_{n+1}) &= -\gamma_n \mu_{n-1}
+ \gamma_{n+1} \mu_{n+1}.
\end{align*}

Now, we multiply the second equation by $\displaystyle \mu_n$ and taking into account the first one we arrive~to
\begin{align*}
- (\xi_n + \xi_{n+1})(\xi_n - \xi_{n+1})
= -\gamma_n \mu_{n-1} \mu_{n} + \gamma_{n+1} \mu_{n} \mu_{n+1}
\end{align*}
and so
\begin{align*}
 \xi_{n+1}^2 - \xi_n^2
= \gamma_{n+1} \mu_{n} \mu_{n+1} -\gamma_n \mu_{n-1} \mu_{n} .
\end{align*}
Hence,
\begin{align}
\label{eq:painleve4_van_assche}
\xi_{n+1}^2 - \xi_0^2 = \gamma_{n+1} \mu_{n} \mu_{n+1} 
&& \text{and} &&
\beta^{\mathsf L}_n \mu_n& = - (\xi_n + \xi_{n+1})
\end{align}
coincide to the ones 
presented in~\cite{boelen_Van_Assche} as discrete Painlev\'e~IV (dPIV) equation.
In fact, taking $\nu_n = \mu_n^{-1}$ we finally arrive to
\begin{align*}
 \nu_{n} \nu_{n+1} 
 = \frac{C_{\mathsf L}\big(\xi_{n+1} - A_{\mathsf L}/2 - nI_N\big)}{\xi_{n+1}^2 - \xi_0^2}
&& \text{and} &&
\xi_n + \xi_{n+1}& = 
\big(C_{\mathsf L}^{-1} B_{\mathsf L} -C_{\mathsf L}^{-1} \nu_n^{-1} \Big) \nu_n^{-1} .
\end{align*}

If we take $B_{\mathsf L} = 0$ in~\eqref{eq:painleve4_van_assche} then 
$\mu_n = C_{\mathsf L} \beta^{\mathsf L}_n$, and so
\begin{align*}
(\beta^{\mathsf L}_n)^2 C_{\mathsf L} = - (\xi_n + \xi_{n+1}) .
\end{align*}
Now, taking square in the first equation in~\eqref{eq:painleve4_van_assche} 
we~get
\begin{align*}
\big(\xi_n + \xi_{n+1}\big)\big(\xi_{n+1} + \xi_{n+2}\big)
= \Big( \big(\xi_{n+1} - A_{\mathsf L}/2 - nI_N\big)^{-1} \big(\xi_{n+1}^2 - \xi_0^2 \big) \Big)^2 ,
\end{align*}
which is an instance of dPIV by Grammaticos, Hietarinta, and Ramani~(cf.~\cite{Grammaticos_Hietarinta_Ramani}).

Thus, \eqref{eq:dPIV_with_B_1} and \eqref{eq:dPIV_with_B_2} for $B_\mathsf L=0_N$ may be considered as non-Abelian extension of this instance of dPIV.

\begin{teo}[\textbf{Non-Abelian extension of the dPIV}]
	When $B_\mathsf L=0_N$, the following nonlocal nonlinear non-Abelian system for the recursion coefficients is fulfilled
\begin{align*}
 & \phantom{ola} \begin{multlined}[t][0.9\textwidth]
 (2n+1) I_N + A_{\mathsf L} 
	+ C_{\mathsf L} (\gamma_{n+1}^{\mathsf L} + \gamma_{n}^{\mathsf L} ))
	+C_{\mathsf L}( \beta_{n}^{\mathsf L})^2 
	\\ = 
	\Big[ \sum_{k=0}^{n-1} \beta_k^{\mathsf L}, C_{\mathsf L}\Big] \sum_{k=0}^{n} \beta_k^{\mathsf L}
	- \Big[\sum_{i,j=0}^{n-1} \beta_i^{\mathsf L} \beta_j^{\mathsf L} - \sum_{k=0}^{n-1}\gamma_k^{\mathsf L}
	, C_{\mathsf L}\Big], 
\end{multlined} 
    \\
 & \phantom{ol} \beta_n^{\mathsf L}-\gamma_{n}^{\mathsf L} \big( C_{\mathsf L}(\beta_n^{\mathsf L} + \beta_{n-1}^{\mathsf L}) \big) + \big( C_{\mathsf L}(\beta_n^{\mathsf L} + \beta_{n+1}^{\mathsf L}) \big) \gamma_{n+1}^{\mathsf L} =
	-\gamma_{n}^{\mathsf L} \Big[ \sum_{k=0}^{n-1} \beta_k^{\mathsf L} , C_{\mathsf L}\Big] +
	\Big[-\sum_{k=0}^{n-1} \beta_k^{\mathsf L} , C_{\mathsf L}\Big] \gamma_{n+1}^{\mathsf L}. 
	\end{align*}
Moreover, this system reduces in the commutative context to the standard dPIV equation.
\end{teo}

\end{document}